\documentclass[12pt, a4paper, parskip=half, abstracton]{scrartcl}
\pdfoutput=1

\usepackage{array}
\usepackage{marginnote}
\usepackage{xcolor}
\usepackage{amscd,amssymb,amsfonts,amsmath,latexsym,amsthm}
\usepackage{hyperref}
\usepackage[capitalize]{cleveref}
\usepackage[all,cmtip]{xy}
\usepackage{mathrsfs}
\usepackage{graphicx}
\usepackage{bm} 
\usepackage{mathtools} 

\usepackage{etoolbox}

\usepackage{chngcntr} 
\usepackage[nottoc]{tocbibind} 
\usepackage{defs_pp1}
\usepackage{slashed}
\usepackage[utf8]{inputenc}
\usepackage{microtype}
\usepackage[english]{babel}

\title{Coarse homology theories and finite decomposition complexity}
\date{\today}
\author{
Ulrich Bunke\thanks{Fakult{\"a}t f{\"u}r Mathematik,
Universit{\"a}t Regensburg,
93040 Regensburg,
Germany\newline
ulrich.bunke@mathematik.uni-regensburg.de} 
\and Alexander Engel\thanks{Fakult{\"a}t f{\"u}r Mathematik,
Universit{\"a}t Regensburg,
93040 Regensburg,
Germany\newline
alexander.engel@mathematik.uni-regensburg.de
	}
\and
Daniel Kasprowski\thanks{
	Rheinische Friedrich-Wilhelms-Universit\"at Bonn, Mathematisches Institut, Endenicher Allee 60,\newline 53115 Bonn, Germany\newline
	kasprowski@uni-bonn.de
}
\and
Christoph Winges\thanks{
		Rheinische Friedrich-Wilhelms-Universit\"at Bonn, Mathematisches Institut, Endenicher Allee 60,\newline 53115 Bonn, Germany\newline
	winges@math.uni-bonn.de}
}

\numberwithin{equation}{section}
\counterwithout{footnote}{section}

\newtheorem{theorem}{Theorem}[section] 
\newtheorem{prop}[theorem]{Proposition}
\newtheorem{lem}[theorem]{Lemma}
\newtheorem{kor}[theorem]{Corollary}
\theoremstyle{remark}
\theoremstyle{definition}
\newtheorem{ddd-alt}[theorem]{Definition}
\newtheorem{ex-alt}[theorem]{Example}
\newtheorem{rem-alt}[theorem]{Remark}

\newenvironment{ddd}    
{%
	\pushQED{\qed}\begin{ddd-alt}}
	{\popQED\end{ddd-alt}}

\newenvironment{ex}    
{%
	\pushQED{\qed}\begin{ex-alt}}
	{\popQED\end{ex-alt}}

\newenvironment{rem}    
{%
	\pushQED{\qed}\begin{rem-alt}}
	{\popQED\end{rem-alt}}

\newcommand{\semi}{\mathrm{semi}}
\newcommand{\free}{\mathrm{free}}
\newcommand{\sub}{\mathrm{sub}}

\newcommand{\VP}{{\mathcal{VS}}}
\newcommand{\FP}{{\mathcal{FS}}}
\newcommand{\SB}{{\mathcal{B}}}

\newcommand{\UBC}{\mathbf{UBC}}

\DeclareMathOperator{\Yo}{Yo}

\newcommand{\BC}{\mathbf{BornCoarse}}

\DeclareMathOperator{\Cofib}{Cofib}

\newcommand{\bA}{{\mathbf{A}}}

\newcommand{\cO}{{\mathcal{O}}}
\newcommand{\cU}{{\mathcal{U}}}
\newcommand{\cY}{{\mathcal{Y}}}

\newcommand{\cD}{{\mathcal{D}}}

 \newcommand{\Cat}{{\mathbf{Cat}}}

\newcommand{\IN}{\mathbb{N}}

\begin{document}
\maketitle

\begin{abstract}
{Using the language of coarse homology theories, we provide an axiomatic account of vanishing results for the fibres of forget-control maps associated to spaces with equivariant finite decomposition complexity.}
\end{abstract}

\tableofcontents
 
\section{Introduction}
  
For a group $G$ we consider the category $G\BC$ of $G$-bornological coarse spaces \cite[Sec. 2]{equivcoarse}. In \cite[Sec. 3 \& 4]{equivcoarse} we introduced the notion of an equivariant coarse homology theory and constructed the universal equivariant coarse homology theory \[\Yo^{s}:G\BC\to G\Sp\cX\]
whose target is the presentable stable $\infty$-category of equivariant coarse motivic spectra. 

To every $G$-bornological coarse space  $X$ one can functorially associate the motivic forget-control map  
\begin{equation}\label{jknkefnhvievevve}
\beta_{X}:  F^{\infty}(X)\to \Sigma F^{0}(X)
\end{equation}
which is a morphism  in $G\Sp\cX$ (see \cite[Def. 11.10]{equivcoarse} and  \eqref{wreofijherofefwefef} below).
 If $E$ is a $\bC$-valued equivariant coarse homology theory, then it factors in an essentially unique way over $G\Sp\cX$ and we get the induced
forget-control map
\[E(\beta_{X}):E(F^{\infty}(X))\to \Sigma E(F^{0}(X))\] for $E$.
The goal of the present paper is to show the following theorem:
\begin{theorem}\label{grkrgoi4334r3f}
Assume:
\begin{enumerate}
\item $E$ is weakly additive (see \cref{fgiw9fuew098fewfewfwf}).
\item $E$ admits weak  transfers (see \cref{rgioggergergergerg}).
\item $\bC$ is compactly generated.
\item $X$ has $G$-finite decomposition complexity ($G$-FDC)   (see \cref{wfiowefwefewfewf}).
\item $G$ acts discontinuously on $X$ (see \cref{weffoiwfewwef423453}).
\end{enumerate}
Then the forget-control map  for $E$ is an equivalence
\[E(\beta_{X}):E(F^{\infty}(X))\to  \Sigma E(F^{0}(X))\ .\]
\end{theorem}

In the case that $X$ is the group $G$ with the canonical bornological coarse structure and {the} action by left multiplication (we often denote this object by $G_{can,min}$), the forget-control map is closely related to the assembly map which appears in Farrell--Jones type conjectures, see \cite[Sec. 11.2]{equivcoarse} for details. 
 \cref{grkrgoi4334r3f}  will be used in \cite{desc} to obtain split-injectivity results for the assembly map. 

\cref{grkrgoi4334r3f} applies to equivariant coarse algebraic $K$-homology $K\bA\cX^{G}$ with coefficients in an additive category $\bA$ with a strict $G$-action, see \cref{wrgiuwrg9gergerger}. 
In this case we {will reprove} Novikov-type results for several classes of groups including a large class of linear groups {in \cite{desc}}.

In the following we give a more detailed description of the result. In 
\cite[Sec. 9]{equivcoarse} we introduced the category of $G$-uniform bornological coarse space
$G\UBC$. This category is related with $G\BC$ by a  forgetful functor 
\[\cF_{\cU}:G\UBC\to G\BC\] (which forgets the uniform structure) and a cone functor
\[\cO:G\UBC\to G\BC\ .\]   These functors are connected by   a natural transformation  
$ \cF_{\cU}\to \cO$ whose motivic   cofibre
\[\cO^{\infty}:=\Cofib(\Yo^{s}\circ \cF\to \Yo^{s}\circ \cO):G\UBC\to G\Sp\cX\]
is called   the germs-at-$\infty$-functor. 
By construction we have the cone fibre sequence  \begin{equation}\label{vrvervevve}
\Yo^{s}\circ \cF_{\cU}\to \Yo^{s}\circ \cO\to \cO^{\infty}\xrightarrow{\partial} \Sigma \Yo^{s}\circ \cF_{\cU}
\end{equation}
of functors from $G\UBC$ to $G\Sp\cX$.

The Rips complex construction provides a functor in the other direction from $G\BC$ to $G\UBC$. 
We let $G\BC^{\cC}$ be the category of pairs $(X,U)$ of a $G$-bornological space
$X$ and an invariant entourage $U$ of $X$.
A morphism $(X,U)\to (X^{\prime},U^{\prime})$ in $G\BC^{\cC}$ is a morphism $f:X\to X^{\prime}$ in $G\BC$ such that
$(f\times f)(U)\subseteq U^{\prime}$. 
Formally, the category $G\BC^{\cC}$ is the Grothendieck construction of the functor $G\BC\to \Cat$ which sends a $G$-bornological coarse space $(X,\cC,\cB)$ to the poset   $\cC^{G}$ of invariant entourages.
In \cite[Sec. 11.1]{equivcoarse} or \cref{def:Rips} we introduce the Rips complex functor
\[P:G\BC^{\cC}\to G\UBC\ , \quad (X,U)\mapsto P_{U}(X)\ .\]
We form the left Kan-extension of the precomposition of the cone sequence \eqref{vrvervevve} with $P$ along the forgetful functor
\[G\BC^{\cC}\to G\BC\ , \quad (X,U)\mapsto X\ .\] The result is a fibre sequence  \begin{equation} \label{wreofijherofefwefef}
F^{0}\to F\to F^{\infty}\xrightarrow{\beta} \Sigma F^{0}
\end{equation}
of functors from $G\BC$ to $ G\Sp\cX$. If we apply this fibre sequence to a bornological coarse space $X$, then we get a fibre sequence in $G\Sp\cX$ whose third morphism is the motivic forget-control map \eqref{jknkefnhvievevve}.

Let $X$ be a $G$-bornological coarse space and $E$ be an equivariant coarse homology theory.
In view of the fibre sequence \eqref{wreofijherofefwefef}, \cref{grkrgoi4334r3f} is an immediate consequence of the following theorem:
\begin{theorem}[\cref{ergijoergregregergeg}]\label{grkrgoi4334r3f1}
Assume:
\begin{enumerate}
\item $E$ is weakly additive.
\item $E$ admits weak  transfers.
\item $\bC$ is compactly generated.
\item $X$ has $G$-FDC.
\item $G$ acts discontinuously on $X$.
\end{enumerate}
Then $E(F (X)) \simeq 0$.
\end{theorem}

Note that the pointwise formula for the left Kan-extension gives
\[E(F(X))\simeq \colim_{S\in \cC^{G}}  E(\cO(P_{U}(X)))\]
whose right-hand side explicitly appears in  \cref{ergijoergregregergeg}. In order to deduce \cref{grkrgoi4334r3f1} 
from \cref{ergijoergregregergeg} furthermore note that the condition that $\bC$ is compactly generated 
implies that phantom objects (\cref{def:phantom}) are trivial. 
  
The notion of finite decomposition complexity was introduced by Guentner, Tessera and Yu \cite{GTY} as a generalization of finite asymptotic dimension. It was used in \cite{Guentner:2010aa} to prove instances of the stable Borel conjecture. Subsequently, Ramras, Tessera and Yu \cite{Ramras:2011aa} employed this notion to prove instances of the $K$-theoretic Novikov conjecture. Kasprowski \cite{kasp} introduced $G$-FDC as an equivariant version of finite decomposition complexity and proved that {if $X$  has  $G$-FDC, then} the forget-control map induces an equivalence
$K{\bA}\cX^{G}(\beta_{X})$
in equivariant coarse algebraic $K$-homology $K{\bA}\cX^{G}$ associated to an additive category~${\bA}$. 

The proofs in \cite{Ramras:2011aa} and \cite{kasp}  use properties of the coarse algebraic $K$-homology functor  $K{\bA}\cX$ which go beyond the four general properties of an equivariant coarse homology listed in \cref{roigwregergerger}.   The main new
contribution of the present paper is the observation 
  that in addition to the  axioms of a coarse homology theory we just need the  following  two additional properties:
\begin{enumerate}
\item existence  of    weak transfers (see \cref{rgioggergergergerg}). \item weak additivity (see \cref{fgiw9fuew098fewfewfwf}).
\end{enumerate}

We show that equivariant coarse ordinary homology $H\cX^{G}$ and equivariant coarse algebraic $K$-homology $K{\bA}\cX^{G}$ associated to an additive category $\bA$ with $G$-action both admit weak transfers and are  weakly additive. Therefore, \cref{grkrgoi4334r3f} applies to these examples and our main result directly implies and generalizes the results in \cite{Guentner:2010aa} and \cite{kasp}.  

\subsection*{Acknowledgements} U.B.~and A.E.~were supported by the SFB 1085 \emph{Higher Invariants} funded by the Deutsche Forschungsgemeinschaft DFG, and A.E.~was further supported by the Research Fellowship EN 1163/1-1 \emph{Mapping Analysis to Homology} of the DFG. D.K.~and C.W.~acknowledge support by the Max Planck Society. Parts of the work presented here were obtained during the Junior Hausdorff Trimester Program \emph{Topology} at the Hausdorff Research Institute for Mathematics (HIM) in Bonn. We also thank Mark Ullmann for helpful discussions.
  
 \section{Equivariant coarse homology theories}
 
 \subsection{Weak transfers}
 We consider a group $G$. In  \cite[Sec.~2]{equivcoarse} we introduced the  category $G\BC$ of $G$-bornological coarse spaces and equivariant proper and controlled maps. The notion of the free union of a family 
  $(X_{i})_{i\in I}$ of $G$-bornological coarse spaces plays an important role in the present paper. 

\begin{ddd}
We define the \emph{free union}
$\coprod_{i\in I}^{\free}X_{i}$ in $G\BC$ as follows:
\begin{enumerate}\item The underlying $G$-set of  $\coprod_{i\in I}^{\free}X_{i}$ is the disjoint union $\coprod_{i\in I}X_{i}$. \item 
A set $B\subseteq \coprod_{i\in I}^{\free }X_{i}$ is bounded if and only if:  \begin{enumerate}\item  The set $|\{i\in I\:|\: B\cap X_{i}\not=\emptyset\}|$ is finite.  \item 
 The set $B\cap X_{i}$ is  bounded (as a subset of $X_{i}$) for all $i$ in $I$.\end{enumerate} \item The coarse structure of $\coprod^{\free}_{i\in I} X_{i}$ is generated by the entourages
$\bigsqcup_{i\in I} U_{i}$ for all families $(U_{i})_{i\in I}$, where $U_{i}$ is an entourage of $X_{i}$. \qedhere
\end{enumerate}
\end{ddd}

 \begin{rem}
Let $(X_{i})_{i\in I}$ be a family of $G$-bornological coarse spaces.
For every $j$ in $I$ we have a canonical morphism $X_{j}\to \coprod_{i\in I}^{\free}X_{i}$. These morphisms induce a morphism \[\coprod_{i\in I} X_{i}\to \coprod_{i\in I}^{\free}X_{i}\ ,\] where
  $\coprod_{i}X_{i}$ denotes the coproduct of the family $(X_{i})_{i\in I}$ in the category $G\BC$.
  The coproduct can be realized such that this morphism is the identity on the underlying sets.
 In general, it is not an isomorphism since the coarse structure on the free union is larger than the coarse structure on the categorical coproduct in $G\BC$, while the bornology of the coproduct is larger than that of the free union.  \end{rem}
 
 Let $\bC$ be a cocomplete stable $\infty$-category and 
 \[E:G\BC\to \bC\] be a functor. The following definition is taken from \cite[Def.~3.10]{equivcoarse}.
\begin{ddd} \label{roigwregergerger}
The functor $E$ is an \emph{equivariant $\bC$-valued coarse homology theory}  if it satisfies the following properties:
	\begin{enumerate}
		\item $E$ is excisive for equivariant complementary pairs.
		\item $E$ is coarsely invariant.
		\item $E$ vanishes on flasque $G$-bornological coarse spaces.
		\item $E$ is {$u$-}continuous.\qedhere
	\end{enumerate}
\end{ddd}

We consider a family $(X_{i})_{i\in I}$ of $G$-bornological coarse spaces and a point $j$ in $I$. We set $I_{j}:=I\setminus \{j\}$.
Then $(X_{j}, \coprod^{\free}_{i\in I_{j}}X_{i})$ is a coarsely excisive decomposition (see \cite[Def.~4.12]{equivcoarse}) of $\coprod^{\free}_{i\in I}X_{i}$.
Since $E$ satisfies excision for coarsely excisive decompositions   \cite[Def. 4.13]{equivcoarse}   we can define a projection
\begin{equation}\label{brtboijogr4g45g5}
p_{j}:E( \coprod^{\free}_{i\in I}X_{i})\simeq E(X_{j})\oplus E( \coprod^{\free}_{i\in I_{j}}X_{i})\to E(X_{j})\ .
\end{equation} 

Let $I$ be a set and $E:G\BC\to \bC$ be an equivariant coarse homology theory.
Then we  define a functor
\[E^{I}:G\BC\to \bC\ , \quad X\mapsto E( \coprod^{\free}_{i\in I}X)\ .\]
For every $j$  in $I$ the projection \eqref{brtboijogr4g45g5}  then  provides a  natural transformation of functors
\[p_{j}:E^{I}\to E \ .\]

\begin{ddd}\label{rgioggergergergerg}
$E$ has {\em weak transfers} for $I$ if there exists a natural transformation
\[\tr_{I}:E\to E^{I}\] such that
\begin{equation}\label{bvtrb4g4tb}
p_{j}\circ \tr_{I}\simeq \id_{E}
\end{equation} for every $j$ in $I$.
\end{ddd}

{\begin{ex}\label{rgeioreggergeg}
	Recall from \cite[Def.~8.8]{equivcoarse}, that equivariant coarse algebraic $K$-homology $K\bA\cX^{G}$ was constructed by assigning to a $G$-bornological coarse space $X$ an additive category $\bV^G_\bA(X)$ of $X$-controlled objects and taking $K$-theory of this category. Let \[(\widehat{K\bA\cX^G})^I:= K(\prod_I\bV^G_\bA(X))\ .\] Sending an object $M$ of $\bV^G_\bA(\coprod_{I}^\free(X))$ to the sequence $(M|_{X_i})_{i\in I}$, where $X_i$ is the $i$th copy of $X$ in the free union, yields a natural transformation \[\tau\colon (K\bA\cX^G)^I\to (\widehat{K\bA\cX^G})^I\ .\] From the definition of the free union, it follows that $\tau$ is a natural equivalence. We then define
	\[\tr_I\colon K\bA\cX^G\xrightarrow{K(\Delta)}(\widehat{K\bA\cX^G})^I\xrightarrow{\tau^{-1}}(K\bA\cX^{G})^I\ ,\]
	where $\Delta$ is the diagonal. Property \eqref{bvtrb4g4tb} for the weak transfer follows from \cite[Rem.~8.16]{equivcoarse}. 
\end{ex}}

{\begin{rem}
	\label{rem:transfers}
	In \cite{coarsetrans} we introduce the notion of a coarse homology theory with transfers. More precisely, we extend the category $G\BC$ to an $\infty$-category $G\BC^Q_{tr}$ whose additional morphisms encode a more general kind of transfers. In \cite{coarsetrans} we show that every coarse homology theory with transfers in particular has weak transfers and we construct extensions of equivariant coarse ordinary homology $H\cX^{G}$ and equivariant coarse algebraic $K$-homology $K\bA\cX^{G}$ to coarse homology theories with transfers.
\end{rem}}
 
 \subsection{Weak additivity}
 
Let $E:G\BC_{tr}\to \bC$ be an  equivariant  $\bC$-valued  coarse homology theory. The following Definition is taken from \cite[Def. 3.12]{equivcoarse}.
\begin{ddd}\label{sdj2398sd}
$E$ is called \emph{strongly additive} if for every family $(X_{i})_{i\in I}$ of $G$-bornological coarse spaces  the morphism
\begin{equation}
\label{iuztrewq}
(p_{j})_{j\in I}:E(\coprod_{i\in I}^{\free}X_{i})\to \prod_{j\in I} E(X_{j})
\end{equation}
     is an equivalence.
 \end{ddd}
 
Many examples of equivariant coarse homology theories are strongly additive. 
The following has been shown in \cite{equivcoarse}.
\begin{theorem}\label{kieorgergergreg}\mbox{}
\begin{enumerate} \item Equivariant coarse ordinary homology $H\cX^{G}$ is strongly additive \cite[Lem.~7.11]{equivcoarse}.
\item  Equivariant coarse algebraic $K$-homology $K\bA\cX^{G}$ with coefficients in  an additive category $\bA$ with a strict $G$-action is strongly additive  \cite[Prop.~8.19]{equivcoarse}.
\end{enumerate}

\end{theorem}
 
The arguments of the present paper use a weaker form of additivity we will now introduce.

Let $\bC$ be a stable $\infty$-category.  In the following an object $K$ of $\bC$ is called compact if it is $\omega$-compact, i.e., the functor $\Map_{\bC}(K,-)$ preserves filtered colimits.
\begin{ddd} \label{def:phantom}
\begin{enumerate} \item 
An object $C$ in $\bC$ is called a {\em phantom object} if $\Map(K,C)\simeq *$ for every compact object $K$ of $\bC$. 
\item A morphism $f:C\to D$ in $\bC$ is called a {\em phantom monomorphism} if and only if 
for every compact object $K$ and morphism
$\phi:K\to C$ the condition $f\circ \phi\simeq 0$ implies $\phi\simeq 0$.
\end{enumerate}
\end{ddd}

\begin{ex}In general, if the fibre of a morphism in $\bC$ is a phantom object, then the morphism is a phantom monomorphism. The converse is not true.
 
If $\bC$ is compactly generated, then  every phantom object in $\bC$ is equivalent to the zero object. This applies e.g. to the stable $\infty$-categories $\Ch_{\infty}$ of chain complexes  and of  spectra $\Sp$.   
\end{ex}

Let $E:G\BC\to \bC$ be an equivariant  $\bC$-valued  coarse homology theory. 
 \begin{ddd}\label{fgiw9fuew098fewfewfwf}
$E$ is called {\em weakly additive} if the morphism
\[ (p_{i})_{i\in I} : E( \coprod_{i\in I}^{\free}X_{i})\to \prod_{i\in I}E(X_{i})\]  is a phantom monomorphism. \end{ddd}

Since every equivalence is a phantom monomorphism it is obvious that a strongly  additive equivariant coarse homology theory is weakly additive. Hence we have the following corollary of \cref{kieorgergergreg}.
\begin{kor}\mbox{}
\begin{enumerate} \item Equivariant coarse ordinary homology $H\cX^{G}$ is weakly additive \cite[Lem.~7.11]{equivcoarse}.
\item  Equivariant coarse algebraic $K$-homology $K\bA\cX^{G}$ with coefficients in  an additive category $\bA$ with a strict $G$-action is weakly additive  \cite[Prop.~8.19]{equivcoarse}.
\end{enumerate}
\end{kor}
 
\section{Finite decomposition complexity} 
 
In this section we introduce the notion of $G$-equivariant finite decomposition complexity ($G$-FDC) for $G$-bornological coarse spaces. Finite decomposition complexity for metric spaces was introduced by Guentner, Tessera and Yu \cite{Guentner:2010aa} and the equivariant version for metric spaces was defined in \cite{kasp}, see also \cite{kas:akt}.
The main theorem of the section is \cref{thm:mainFDC} which can be interpreted as a version of \cite[Prop.~6.6]{Ramras:2011aa} for general weakly additive coarse homology theories with weak transfers.

\subsection{Finite decomposition complexity}

Let $G$ be a group. 
Recall 
 that a $G$-coarse space is a coarse space $(X,\cC)$ such that the subset $\cC^{G} $ of invariant entourages   is cofinal in $ \cC$. 

 Let   $T$ be an invariant entourage of a  $G$-coarse space $X$.      Furthermore, let   $U$ and $V$ be $G$-invariant subsets of $X$. 
\begin{ddd}
	The subsets $U,V$ are said to be \emph{$T$-disjoint} if \[T[U]\cap V=\emptyset \quad \mbox{and} \quad U\cap T[V]=\emptyset\ .\qedhere\]
\end{ddd}

 Let $X$ be a $G$-set. \begin{ddd} An {\em equivariant family} of subsets of $X$ is a family of subsets
$(U_{i})_{i\in I}$ indexed by a $G$-set $I$  such that $U_{g(i)}=g(U_{i})$  for every $i$ in $I$ and $g$ in $G$. \end{ddd}

 Let $Y$ be a $G$-set, $X$ be a $G$-coarse space with coarse structure $\cC$, and  $f:Y\to X$ be an equivariant map of sets.
 \begin{ddd} The induced $G$-coarse structure $f^{-1}\cC$ 
on $Y$ is the maximal $G$-coarse structure on $Y$ such that the map $f$ is controlled.
  \end{ddd}
  
Let $X$ be a $G$-set  and $\cX:=(X_{i})_{i\in I}$ be a partition of $X$ into   $G$-invariant subsets.
Then we form the $G$-invariant entourage 
  \[U(\cX):=\bigsqcup_{i\in I} X_{i}\times X_{i}\ .\]
  If $\cC$ is a $G$-coarse structure on $X$, then we define a new $G$-coarse structure
  \[\cC(\cX):=\cC\langle \{V\cap U(\cX)\:|\: V\in \cC\}\rangle\ .\] 
Let $X$ be a $G$-coarse space with coarse structure $\cC$ and $\cU:=(U_i)_{i \in I}$ be an equivariant family of  subsets of $X$. Then we have a natural map of $G$-sets $f:\coprod_{i \in I} U_i \to X$.

\begin{ddd} \label{ioefoweifwefwef} We define the $G$-coarse space  $ \coprod^{\sub}_{i \in I} U_i$ to be the $G$-set  
  $\coprod_{i \in I} U_i$ with the $G$-coarse structure $(f^{-1}\cC)(\cU)$.
   \end{ddd}
   Note that the $G$-coarse structure $(f^{-1}\cC)(\cU)$ is
   generated by the entourages
  \[\sub^{\cU}(V):=V\cap U(\cU)\] for all entourages $V$ of the induced coarse structure  $ f^{-1}\cC $. 
  
  \begin{rem} \label{wrowpfewfwf}Note that
   the set of entourages
   $\{\sub^{\cU}(V)\:|\: V\in \cC^{G}\}$  of $\coprod^{\sub}_{i \in I} U_i$
   is cofinal in the coarse structure of $\coprod_{i \in I}^{\sub} U_i $.
   \end{rem}
   
   \begin{rem}  
   The coarse space  $\coprod_{i \in I}^{\sub} U_i $ is a coarsely disjoint union of  the subsets  $U_{i}$ which all have the coarse structure induced from $X$ via the inclusion $U_{i}\to X$.  But the union is in general not free.
   
   There is a coarse map
   $\coprod_{i \in I}^{\sub} U_i\to X$, but the coarse structure on the domain is in general smaller than the one induced from $X$.   
   \end{rem}

Let $X$ be a $G$-bornological coarse space and $(U_{i})_{i\in I}$ an equivariant family of subsets of~$X$.   
\begin{ddd}
We say that the family $(U_{i})_{i\in I}$ is \emph{nice} if for every invariant entourage $S$  of  $X$ containing the diagonal  the natural morphism
\[\coprod^{\sub}_{i\in I} U_{i}\to \coprod^{\sub}_{i\in I} S[U_{i}]\]
is an equivalence of $G$-bornological coarse spaces.
\end{ddd}
   
   \begin{rem}
   In the non-equivariant case (i.e., if $G$ is the trivial group)  every family of subsets of a coarse space is nice. In contrast, for non-trivial groups the inclusion of an invariant subset of   $G$-coarse space into its thickening   need not be a coarse equivalence. 
   \end{rem}

Let $X$ be a $G$-coarse space and let $\cF$ be a class of $G$-coarse spaces.
\begin{ddd}
We say that $X$ is \emph{decomposable over $\cF$} if for every invariant entourage $T$ of $X$ there exist pairwise $T$-disjoint and nice equivariant families $(U_{i}^{T})_{i\in I}$ and $(V_{j}^{T})_{j\in J}$ of subsets of $X$ such that
 \begin{enumerate}
  \item $X=U^T\cup V^T$ with $U^T := \bigcup_{i \in I} U_{i}^{T}$ and $V^T := \bigcup_{j \in J} U_{j}^{T}$.
  \item The $G$-coarse spaces
   $\coprod_{i\in I}^{\sub}U_i^T$ and $ \coprod_{j\in J}^{\sub}V_j^T$
  belong to $\cF$. \qedhere
 \end{enumerate}
\end{ddd}
Let $\cF$ be a  class of $G$-coarse spaces.
\begin{ddd}
We say that the class   $\cF$ is \emph{closed under decomposition} if every $G$-coarse space that is decomposable over $\cF$ is  contained in $\cF$.
\end{ddd}

\begin{rem}
Recall that the $G$-coarse structure $\cC$ of a $G$-coarse space $X$ induces a $G$-invariant equivalence relation $\bigcup_{U\in \cC} U$ on $X$.  The equivalence classes for this relation  are called coarse components. The group $G$ acts on the set of coarse components $\pi_{0}(X)$ in the natural way.
 
Furthermore recall that a subset $Y$ of a coarse space $X$ is called $T$-bounded for an entourage $T$ of $X$ if $Y\times Y\subseteq T$. 
 \end{rem}
 
 Let $X$ be a $G$-coarse space.
\begin{ddd}\label{efiuweofewfewf}
   $X$  is called \emph{semi-bounded} if there exists an invariant entourage $T$ of $X$ such that every coarse component of $X$ is $T$-bounded. \end{ddd} 
We denote the class of semi-bounded $G$-coarse spaces by $\SB$.

\begin{ddd}
 Let $\cD$ be the smallest class of $G$-coarse spaces that contains $\SB$ and is closed under decomposition. 
\end{ddd}

 Let $X$ be a $G$-coarse space.
 
 \begin{ddd}\label{wfiowefwefewfewf} $X$ has \emph{$G$-equivariant finite decomposition complexity} (for short \emph{$G$-FDC}) if it is contained in $\cD$.
\end{ddd}

Let $X$ be a $G$-coarse space with coarse structure $\cC$. Then $X$ has a minimal bornology $\cB_{\cC}$
which is compatible with the coarse structure. It is generated by the subsets $S[x]$ for all $x$ in $X$ and $S$ in $\cC$.

\begin{ddd}\label{weffoiwfewwef423453}
	 The  $G$-action on $X$ is  said to be \emph{discontinuous} if  for every   point  $x$   of  $ X$ and every   $B$ in $\cB_{\cC}$   the intersection  $Gx \cap B$ is finite. \end{ddd}

\subsection{\texorpdfstring{$G$}{G}-FDC for bornological coarse spaces and the main theorem}
Let $X$ be a $G$-bornological coarse space.
\begin{ddd}
We say $X$ has $G$-FDC if its underlying $G$-coarse space has $G$-FDC.
\end{ddd}
 
We refer to \cite[Def.~9.9]{equivcoarse} for the definition of the category of $G$-uniform bornological coarse spaces.

Let $X$ be a $G$-bornological coarse space, $Y$ be a $G$-invariant subset,  and $S$ be an invariant entourage.
 \begin{ddd} \label{def:Rips}
	We define the $G$-uniform bornological coarse space $P^{X}_{S}(Y) $  as follows:
	\begin{enumerate} 
		\item The underlying $G$-set  of $P^{X}_{S}(Y)$ is the $G$-set of probability measures on $X$ whose support is finite $S$-bounded and contained in $Y$. 
		\item\label{fiufowefewfefwe} We consider $P^{X}_{S}(X)$ as a $G$-simplicial complex with the  $G$-equivariant spherical quasi-metric.   This quasi-metric induces the $G$-uniform structure on the subset $P^{X}_{S}(Y)$.
		\item  The quasi-metric from \ref{fiufowefewfefwe} also induces the $G$-coarse structure on the subset $P^{X}_{S}(Y)$.
		\item The bornology of $P^{X}_{S}(X)$ is generated by the subsets $P^{X}_{S}(B)$ for all bounded subsets $B$ of $X$. It induces the bornology of $P^{X}_{S}(Y)$.
	\end{enumerate}
	We call $P_{S}^{X}(Y)$ the \emph{Rips complex}.
\end{ddd} 

In order to abbreviate the notation we will simplify the notation and write
\[P_{S}(X):=P_{S}^{X}(X)\ .\]

\begin{rem}
Let $X$ be a $G$-bornological coarse space. If $Y$ is a $G$-invariant subset, then we will use the notation $Y$ also for the $G$-bornological coarse space obtained by equipping this subset with the coarse structure  and the bornological structure  induced from $X$. If we want to underline that the structures come from $X$, then we also use the more precise notation $Y_{X}$ for this  
 $G$-bornological coarse space.

 If $S$ is an invariant entourage of $X$ and we set \begin{equation}\label{gergjoigrgergreg}
S_{Y}:=(Y\times Y) \cap S\ ,
\end{equation} then   
 we have a natural morphism of $G$-uniform bornological coarse spaces
 \[P_{S_{Y}}(Y)\to P^{X}_{S}(Y)\ .\]
It is  an isomorphism of the underlying bornological spaces.
But the  quasi-metric on the left-hand side might be larger than the quasi-metric on the right-hand side.    
\end{rem}

The following is the main theorem of the present section: 
Let $X$ be a $G$-bornological coarse space with coarse structure $\cC$ and  {let $E$ be a weakly additive coarse homology theory with weak transfers.}
\begin{theorem}\label{ergijoergregregergeg}
	\label{thm:mainFDC}
	If $X$ has $G$-FDC and a discontinuous $G$-action,
	then
	\[\colim_{S\in \cC^G}E(\cO(P_{S}(X)))\]
	is a  phantom object. 
\end{theorem}

\begin{ex}\label{wrgiuwrg9gergerger} In view of \cref{rgeioreggergeg}, \cref{rem:transfers}  and \cref{kieorgergergreg}, \cref{ergijoergregregergeg} applies to equivariant coarse ordinary homology $H\cX^{G}$ and to equivariant coarse algebraic $K$-homology $K\bA \cX^{G}$ with coefficients in an additive category $\bA$ with a strict $G$-action.
\end{ex}

The remaining  part of this section is devoted to the proof of \cref{thm:mainFDC}.
We fix once and for all a $\bC$-valued weakly additive (see \cref{fgiw9fuew098fewfewfwf}) coarse homology theory $E$ with weak transfers (see \cref{rgioggergergergerg}).

\subsection{Structure of the proof}\label{efwoifuiweofewfewff}
Let $X$ be a $G$-bornological coarse space with $G$-coarse structure $\cC$.
 \begin{ddd}\label{opwefwefewfwf}
	We call $X$ \emph{$E$-vanishing} if for every $G$-invariant subset $Y$ of $X$
	\[\colim_{S\in \cC^G} E(\cO(P_{S_{Y}}(Y)))\]
	is a phantom object of $\bC$.\end{ddd}
 Let $\cV_{E}$ denote the class of $E$-vanishing $G$-bornological coarse spaces.

\begin{rem}\label{foiejifoefff23f23} Let $Y$ be a $G$-invariant subset of $X$ and let $\cC(Y)$ denote the coarse structure  of $Y_{X}$. Then we have an equality of sets   $\{S_{Y}\:|\: S\in \cC^{G}\}=\cC^{G}(Y) $  (see \eqref{gergjoigrgergreg} for notation) of entourages of $Y$. Hence we have an equivalence  of objects of $\bC$
\[\colim_{S\in \cC^G} E(\cO(P_{S_{Y}}(Y)))\simeq 
\colim_{T\in \cC^G(Y)}E(\cO(P_{T}(Y)))\ .\]
We could define a notion of  a weakly $E$-vanishing $G$-bornological coarse space $X$ by just requiring that \[\colim_{S\in \cC^G} E(\cO(P_{S}(X)))\] is a  phantom object. Then $X$ is $E$-vanishing if and only if all its $G$-invariant subsets are weakly vanishing.

The motivation to define the notion of $E$-vanishing $G$-bornological coarse spaces 
as above is that it better suits the induction arguments below.
\end{rem}

 {\cref{thm:mainFDC} follows from the next theorem. Since every equivariant subspace of a $G$-bornological coarse space with $G$-FDC has again $G$-FDC, both theorems are actually equivalent.}
 
 \begin{theorem} \label{fopwefewfewfw}The class of $G$-bornological coarse spaces with $G$-FDC and discontinuous $G$-action is contained in $\cV_{E}$.  \end{theorem}  

To prove \cref{fopwefewfewfw}  it suffices to show that the class $\cV_{E}$ contains all semi-bounded $G$-bornological coarse spaces with discontinuous $G$-action (\cref{prop:semiboundedisvanishing}) and is closed under decomposition. We will proceed as follows:

In \cref{roigrgegg} we introduce  the notion of an $E$-vanishing sequence  of $G$-bornological coarse spaces. Furthermore, in \cref{def:strongdec} we define   the concept of  decomposability of sequences of $G$-bornological coarse spaces. We let $\VP_{E}$ denote the class of $E$-vanishing sequences.
The main steps of the proof are now as follows:
\begin{enumerate}
\item\label{eiowfowefewfew} If a sequence     of $G$-bornological coarse spaces is decomposable over $\VP_{E}$, then the sequence belongs to $\VP_{E}$ (\cref{thm:dec}). 
\item \label{foiwefweweff3ee} If a $G$-bornological coarse space $X$ is decomposable over $\cV_{E}$, then the constant sequence $\underline{X}$ is decomposable over $\VP_{E}$ (\cref{cor:constantdecomposableseq}).
\item\label{ifuewoiufiowefwefwefwf} By \ref{eiowfowefewfew} and \ref{foiwefweweff3ee}, if $X$ is decomposable over $\cV_{E}$, then $\underline{X}$ belongs to $\VP_{E}$.
\item \label{reuifzewiewfwefewfew} If $X$ is a  $G$-bornological coarse space and $\underline{X}$ belongs to $\VP_{E}$, then $X$ belongs to $\cV_{E}$ (\cref{lem:constantseqvanishing-componentvanishing}).
 \item By \ref{ifuewoiufiowefwefwefwf} and \ref{reuifzewiewfwefewfew} we can conclude that
 if $X$ is  decomposable over $\cV_{E}$, then it belongs to $\cV_{E}$.
 \end{enumerate}

\subsection{Vanishing pairs and strong decomposability}

Let $[0,\infty)_{d}$ denote the positive ray in $\R$ considered as a $G$-bornological coarse space with the trivial $G$-action and the bornological coarse structure induced from the standard metric.

Let $X$ be a $G$-uniform bornological coarse space with coarse structure $\cC$. Recall that the  cone $\cO(X)$ is obtained from the $G$-bornological coarse space $[0,\infty)_{d}\otimes X$ by
replacing the coarse structure $\cC_{\otimes} $ of this tensor product by the hybrid structure $\cC_{h}$. In particular we have
$\cC_{h}\subseteq \cC_{\otimes}$. It follows that
the projection
 $\pr:[0,\infty)\times X\to X$ is a morphism $([0,\infty)\times X,\cC_{h})\to (X,\cC)$ in the category of $G$-coarse spaces. 
 
 For $t$ in $[0,\infty)$ and $x$ in $X$ we will denote the corresponding point in $\cO(X)$ by $(t,x)$.

Let $X$ be a $G$-uniform bornological coarse space whose coarse structure is induced by a metric $d$. We set  \[U_{r}:=\{(x,y)\in X\times X\:|\: d(x,y)\le r\}\ .\] Let $R$ be an entourage of $\cO(X)$.
\begin{ddd}\label{iofjoewffewfw}
We define the \emph{propagation} $l(R)$ in $[0,\infty)$ by
\[l(R):=\inf\{r\in [0,\infty]\:|\: (\pr\times \pr)(R)\subseteq U_{r}\}\ .\qedhere\]
 \end{ddd}

We consider a family $(X_{i})_{i\in I}$ of $G$-uniform bornological coarse spaces such that the coarse structure of $X_{i}$ is induced by an invariant metric. Here we consider the metric as specified, but we  will not introduce special notation for it. Recall that every  entourage $R$ of  
  the free union  $\coprod_{i\in I}^{\free}\cO(X_i)$ is given by $\bigsqcup_{i\in I}R_{i}$
  for a uniquely determined family 
  $(R_{i})_{i\in I}$, where $R_{i}$ is an entourage of $X_{i}$. 
  We define the \emph{propagation} of $R$ to be the element \[r(R):= \sup_{i\in I} l(R_{i})\] in $[0,\infty]$.
  
  \begin{ddd}\label{regeriogregregerg}
For $r$ in $(0,\infty)$ we define $\coprod_{i\in I}^{\semi(r)}\cO(X_i)$ to be the $G$-bornological coarse space given as follows:
\begin{enumerate}
\item The underlying  $G$-bornological space is the one of $\coprod_{i\in I}^{\free}\cO(X_i)$.
\item The $G$-coarse structure is generated by 
the entourages $R$ of $\coprod_{i\in I}^{\free}\cO(X_i) $ with propagation satisfying $r(R)\le r$.
\end{enumerate}
We further set 
\[\coprod_{i\in I}^{\semi}\cO(X_i):=\colim_{r\in (0,\infty)}\coprod_{i\in I}^{\semi(r)}\cO(X_i)\ .\qedhere\]
 \end{ddd}
 Note that the structure maps of the colimit above are given by the identity of the underlying set.
 
 \begin{rem}In other words, the coarse structure of $ \coprod_{i\in I}^{\semi}\cO(X_i)$
 is generated by those entourages of the free union which in addition have finite propagation.
 \end{rem}
 We have   canonical morphisms \[\coprod_{i\in I}^{\semi(r)}\cO(X_i)\to\coprod_{i\in I}^{\semi}\cO(X_i)\to  \coprod_{i\in I}^{\free}\cO(X_i)\ .\]
 \begin{ex}
Let $X$ be a $G$-bornological coarse space, $Y$ be an invariant subset of $X$, and $S$ be an invariant entourage of $X$. By \cref{def:Rips}   the Rips complex $P^{X}_{S}(Y)$ is a $G$-uniform bornological coarse space with a specified invariant metric.  This is our main source of examples.
\end{ex}
 
In the following we will use the abbreviated notation $(X_{n})$ for sequences $(X_{n})_{n\in \nat}$ of bornological coarse spaces  indexed by $\nat$.  If not said differently, by $\cC_{n}$ we denote the coarse structure of $X_{n}$.

Let $(X_n) $ be a sequence of $G$-bornological coarse spaces and let $(Y_n) $ be a sequence of $G$-invariant subsets, i.e., we have $Y_n \subseteq X_n$. Let $\cC_{n}^{G}$ denote the partially ordered set of $G$-invariant entourages of $X_{n}$ and consider a family $(S_{n})$ in  $ \prod_{n\in \nat}\cC^{G}_{n}$.
  \begin{lem}
	\label{lem:metricsonrips}

	For $d $ in $\IN$ the inclusion of the underlying $G$-bornological spaces  defines a morphism of $G$-bornological coarse spaces 
	\[  \coprod_{n\in\IN}^{\semi(d)} \cO(P_{S_n}^{X_n}(Y_n)) \to \coprod_{n\in\IN}^{\semi}\cO(P_{(S_n^d)_{Y_{n}}}(Y_n)). \]
	
	Furthermore, the canonical map
	\[\colim_{(S_n) \in   \prod_{n\in \nat}\cC^{G}_{n}}E\big(\coprod_{n\in\IN}^{\semi}\cO(P_{(S_n)_{Y_{n}}}(Y_n)) \big)\to \colim_{(S_n) \in   \prod_{n\in \nat}\cC^{G}_{n}} E\big(\coprod_{n\in\IN}^{\semi}\cO(P_{S_n}^{X_n}(Y_n)) \big) \]
	is an equivalence.
\end{lem}
\begin{proof}
	The second claim follows from the first by $u$-continuity of $E$. 
		
Assume that the sequence $(R_{n}) $ defines an entourage of the free union with propagation bounded by $d$. Fix a natural number $n$ and consider $(t,y)$ and $(t^{\prime},y^{\prime})$ in $\cO(P_{S_{n}}^{X_{n}}(Y_{n}))$.
	If $((t,y),(t',y'))$ belongs to $R_n$, then the distance between $y$ and $y'$ in $P_{S_n}^{X_n}(Y_n)$ is bounded by $d$.  There exists a family of points
	$(x_{i})_{i=0}^{d}$ in $X_{n}$ such that $x_{0}$ is a vertex of the simplex containing $y$ and $x_{d}$ is a vertex of a simplex containing $y^{\prime}$, and for every $i$ in $ \{0,\dots,d-1\}$ we have $(x_{i},x_{i+1})\in S_{n}$. Then $x_{0}$ and $x_{n}$ belong to $Y_{n}$, and the distance between $x_{0}$ and $x_{d}$ in $P_{(S_{n}^{d})_{Y_{n}}}(Y_{n})$ is   bounded by $1$. Consequently,
	the distance between $y$ and $y^{\prime}$ in $P_{(S^{d}_{n})_{Y_{n}}}(Y_{n})$ is bounded by $3$. 	
\end{proof}
	\begin{rem}
	The proof shows that we actually have a morphism of $G$-bornological coarse spaces \[  \coprod_{n\in\IN}^{\semi(d)} \cO(P_{S_n}^{X_n}(Y_n)) \to \coprod_{n\in\IN}^{\semi(3)}\cO(P_{(S_n^d)_{Y_{n}}}(Y_n)). \qedhere\]
	\end{rem}

Let $(X_{n}) $ be a sequence of $G$-bornological coarse spaces and $(S_{n})$ and $(S^{\prime}_{n})$ be families in $ \prod_{n\in \nat}\cC^{G}_{n}$ such that $(S_{n})\le (  S_{n}^{\prime})$.  
Then we have a commuting square of morphisms of $G$-bornological coarse spaces:
  
 \[\xymatrix{\coprod\limits^{\semi}\limits_{n\in\IN}\cO(P_{S_n}(X_n)) \ar[r]\ar[d]& \coprod\limits^{\free}\limits_{n\in\IN}\cO(P_{S_n}(X_n))\ar[d]\\\coprod\limits^{\semi}\limits_{n\in\IN}\cO(P_{S^{\prime}_n}(X_n)) \ar[r] & \coprod\limits^{\free}_{n\in\IN}\cO(P_{S^{\prime}_n}(X_n)) }\]
  
 In the following lemma we consider the colimit in the vertical direction. 
\begin{lem}
	\label{lem:comparison}
	We have an equivalence
	\[\colim_{(S_n)\in \prod_{n\in \nat}\cC^{G}_{n }}E\big(\coprod^{\semi}_{n\in\IN}\cO(P_{S_n}(X_n))\big)\to \colim_{(S_n)\in \prod_{n\in \nat}\cC^{G}_{n }}E\big(\coprod^{\free}_{n\in\IN}\cO(P_{S_n}(X_n))\big).\]
\end{lem}
\begin{proof}
We produce an inverse equivalence. Let $R$ be an entourage of the free union associated to the family $(R_n) $. Then by the same argument as in the proof of \cref{lem:metricsonrips} the inclusion of the underlying sets induces a morphism of bornological coarse spaces. 
\[ (\coprod^{\free}_{n\in\IN}\cO(P_{S_n}(X_n)))_R \to \coprod^{\semi(3)}_{n\in\IN} \cO(P_{S_n^{l(R_n)}}(X_n))\ .\]
 By $u$-continuity of $E$ applied to the colimit over the entourages $R$ of the free union
 we get the desired inverse.
\end{proof}
Let $(X_n)$ be a sequence of $G$-bornological coarse spaces and let $(Y_n) $ be a sequence of $G$-invariant subsets. Then we  consider the   object $\bC^\infty((Y_n))$  of $\bC$ given by 
 \begin{equation}\label{gerg43r4r3} \bC^\infty((Y_n)):=\colim_{N\in\IN}\colim_{(S_n) \in   \prod_{n \in \nat}\cC^{G}_{n}} E\big(\coprod_{N \leq n}^{\semi}\cO(P_{(S_n)_{Y_{n}}}(Y_n))\big)
 \ .\end{equation}
\begin{rem} Note that the connecting map for the outer  colimit   involves the projection
 \[E\big(\coprod_{N \leq n}^{\semi}\cO(P_{(S_n)_{Y_{n}}}(Y_n))\big)\to E\big(\coprod_{N+1 \leq n}^{\semi}\cO(P_{(S_n)_{Y_{n}}}(Y_n))\big)\]
given by excision for $E$ along the coarsely excisive pair 
\[\big(\cO(P_{(S_{N})_{Y_{N}}}(Y_{N})), \coprod_{N+1 \leq n}^{\semi}\cO(P_{(S_n)_{Y_{n}}}(Y_n))\big)\ .\qedhere\]
\end{rem}
 \begin{rem}
 In view of the first part of \cref{foiejifoefff23f23} we could rewrite the definition of $\bC^{\infty}((Y_{n}))$ as follows: \[\bC^\infty((Y_n)):=\colim_{N\in\IN}\colim_{(T_n) \in   \prod_{n \in \nat}\cC^{G}_{n}(Y_{n})}E\big(\coprod_{N \leq n}^{\semi}\cO(P_{T_n}(Y_n))\big) \ .\]
 In particular, the object $\bC^{\infty}((Y_{n}))$ of $\bC$ is an intrinsic invariant of the sequence of $G$-bornological coarse spaces $(Y_{n})$. 
 \end{rem}

  We furthermore define  the free version  \[ \bC^\infty_\free((Y_n)):=\colim_{N\in\IN}\colim_{(T_n) \in   \prod_{n \in \nat}\cC^{G}(Y_{n})} E\big(\coprod_{N \leq n}^{\free}\cO(P_{T_n}(Y_n))\big) \ .\]
Then \cref{lem:comparison} has the following consequence:
\begin{kor}\label{ef988u9u89423r}
The natural morphism $\bC^{\infty}((Y_{n}))\to \bC^{\infty}_{\free}((Y_{n}))$ is an equivalence.
\end{kor}
Let $(X_n)$ be a sequence of $G$-bornological coarse spaces. 

 \begin{ddd}\label{roigrgegg}
 The sequence 
 $(X_n) $  is called \emph{$E$-vanishing}  if for every sequence $(Y_n) $ of $G$-invariant subsets   the object  $\bC^\infty((Y_n))$ is a  phantom object.
\end{ddd}
 We let $\VP_{E}$ denote the class of $E$-vanishing sequences.

 Let $(X_n)$ and $(Y_n)$ be sequences of $G$-bornological coarse spaces. 
\begin{lem}\label{lem:vanishingiscoarselyinvariant}
	Assume that  for  every natural number $n$ the bornological coarse space $X_n$ is  equivalent to $Y_n$. Then the sequence
 $(X_n)$ is $E$-vanishing if and only if the sequence $(Y_n)$ is $E$-vanishing.
\end{lem}
\begin{proof}
By the symmetry of the assertion it suffices to show that if  $(Y_{n})$ is $E$-vanishing, then also $(X_{n})$ is $E$-vanishing.

By assumption, for every natural number $n$ we can find morphisms of $G$-bornological coarse spaces $f_{n}:X_{n}\to Y_{n}$ and $g_{n}:Y_{n}\to X_{n}$  and entourages $U_{n}$ of $X_{n}$ and $V_{n}$ of $Y_{n}$  
 such that $g_n\circ f_n$  is $U_{n}$-close to $\id_{X_{n}}$, and 
 $f_n\circ g_n$ is $V_{n}$-close to $\id_{Y_{n}}$.

 Assume that $S_{n}$ in  $\cC^G_{n}$ is given. If we choose    $S_{n}^{\prime}$ in $\cC^G(X_n)$   such that
 \[((g_{n}\circ f_{n})\times (g_{n}\circ f_{n}))(S_{n})\cup U_{n}\subseteq S_{n}^{\prime}\ ,\]
 then $g_n\circ f_n$ induces a map $P_{S_{n}}(X_{n})\to P_{S_{n}^{\prime}}(X_{n})$
 which is homotopic to $\id_{X_{n}}$ by a unit-speed homotopy. We have a similar statement for the composition $f_{n}\circ g_{n}$. By homotopy invariance  of the cone functor \cite[Cor. 9.38]{equivcoarse}  and a cofinality consideration this implies that $f_{n}$ induces an equivalence  \begin{equation}\label{qdiqwzdiqwduqwdwqdq}
 \colim_{S\in \cC^{G}_{n}}E( \cO(P_{S}(X_{n})))\simeq  \colim_{T\in \cC^{G}(Y_{n})}E(\cO(P_{T}(Y_{n})))\end{equation} in $\bC$
for every natural number $n$.

We assume that $(Y_{n})$ is $E$-vanishing. Then by \cref{roigrgegg}  the object
$\bC^{\infty}((Y_{n}))$   is a   phantom object.
We show    that $\bC^{\infty}((X_{n}))$ is a phantom object, too. 

Let $K$ be a compact object in $\bC$   and $\phi:K\to \bC^{\infty}((X_{n}))$ be a morphism. We must show that $\phi$ is equivalent to zero.

By compactness of $K$ there is a factorization $\tilde \phi$ as in the following diagram:
\[\xymatrix{&E(\coprod\limits_{N\le n}^{\free}  \cO(P_{S_{n}}(X_{n})))\ar[d]\ar[dr]^{!!}\ar@{-->}[r]&E(\coprod\limits_{N\le n}^{\free}  \cO(P_{T_{n}}(Y_{n})))\ar[dr] &\\K\ar@{..>}[urr]\ar@/^-1cm/[rr]^-{!}\ar[ur]^-{\tilde \phi}\ar[r]_-{\phi}&\bC^{\infty}((X_{n}))&\prod\limits_{N\le n}E( \cO(P_{S_{n}}(X_{n})))\ar@{-->}[r]&\prod\limits_{N\le n}E( \cO(P_{T_{n}}(Y_{n})))}\]
If we choose  the sequence of entourages $(T_{n})$ sufficiently large (depending on the choice of $(S_{n})$), then the dashed arrows exist.
Note that by \cref{ef988u9u89423r} and our assumption we know that $\bC^{\infty}_{\free}((Y_{n}))$ is a phantom object. Since $K$ is compact, after increasing $N$ further and choosing  $(T_{n})$ sufficiently large  the dotted arrow becomes equivalent to zero. Now we use \eqref{qdiqwzdiqwduqwdwqdq} again  in order to see that if we choose $(S_{n})$ and $(T_{n})$ sufficiently  large, then the arrow marked by $!$ is equivalent to zero.
 But then $\tilde \phi$ is equivalent to zero  since we assume that $E$ is weakly additive and hence the arrow marked by $!!$ is a phantom monomorphism. This finally implies that $\phi$ is equivalent to zero.\end{proof}

If $X$ is a $G$-bornological coarse space, then we can consider the constant sequence $\underline{X}$ of $G$-bornological spaces indexed by $\nat$.

Let $X$ be a $G$-bornological coarse space with coarse structure $\cC$. 
\begin{lem}\label{lem:constantseqvanishing-componentvanishing}
If $\underline{X}$ is an $E$-vanishing sequence, then $X$ is $E$-vanishing.
\end{lem}
\begin{proof}
We must show that for every $G$-invariant subset $Y$ of $X$ the object \begin{equation}\label{wfoihiofefwefwefewfewf}
\colim_{S\in \cC^{G}}E(\cO(P_{S_{Y}}(Y)))\end{equation}
is a  phantom object.

 Let $K$ be a compact   object of $\bC$ and  consider a morphism \[f \colon K \to \colim_{S \in \cC^G} E(\cO(P_{S_{Y}}(Y)))\ .\] We must show that $f$ is equivalent to zero.
	
	 In the following we build step by step the following diagram: 
	\[\xymatrix{&K\ar@{..>}[ddddr]\ar@{.>}[d]^{\tilde f}\ar[r]^-{f}\ar@/_3.5cm/@{-->}[dddd]_-{f'}& \colim\limits_{S \in \cC^G} E(\cO(P_{S_{Y}}(Y)))\ar[d]^-{\tr_{\nat}}\\
	   & E(\cO(P_{R_{Y}}(Y)))\ar[ru]\ar[d]^-{\tr_{\nat}} &  \colim\limits_{S \in \cC^G} E(\coprod\limits_{n \in \IN}^{\free}\cO(P_{S_{Y}}(Y)))\ar[d]^{!}\\
	   &E(\coprod\limits_{n \in \IN}^{\free} \cO(P_{R_{Y}}(Y)))\ar[d]^{!!!}\ar[ru] & \colim\limits_{(S_n)\in \prod\limits_{n\in \nat}\cC^{G}} E(\coprod\limits_{n \in \IN}^{\free} \cO(P_{S_n}(Y))) \\
	   &E(\coprod\limits_{n \in \IN}^{\free} \cO(P_{(R_n)_{Y}}(Y)))\ar[ur] & \\
	   &E(\coprod\limits_{n \leq N_0}^{\free} \cO(P_{(R_n)_{Y}}(Y)))\ar[u]\ar[r] & \colim\limits_{N\in \nat} \colim\limits_{ (S_n)\in \prod\limits_{n \in \nat}\cC^{G}}E(\coprod\limits_{n \leq N}^{\free} \cO(P_{(S_n)_{Y}}(Y)))\ar[uu]_{!!} \\
	}\]
	First of all, by compactness of $K$ there exists an invariant entourage $R$ of $X$   such that we have the factorization $\tilde f$. The next square expresses the naturality of the transfer $\tr_\nat$ (see \cref{rgioggergergergerg}).
	The map marked by $!$ is the canonical map induced by the inclusion of the index set of the colimit in the domain into the index set of the colimit of its target.
	
	 By \cref{ef988u9u89423r} for every $G$-invariant subset $Y$ of $X$ we have an equivalence
\begin{equation}\label{veviuhiufrerf}
\bC^{\infty}(\underline{Y})\simeq \bC^{\infty}_{\free}(\underline{Y})\ .
\end{equation}
 By assumption,
$\bC^{\infty}(\underline{Y})$ and hence by \eqref{veviuhiufrerf} also 
$\bC^{\infty}_{\free}(\underline{Y})$ is a  phantom object. Since the latter is the cofibre of the lower right vertical map marked by $!!$, and since 
  $K$ is compact,  we get the diagonal dotted factorization of the composition $!\circ \tr_{\nat}\circ f$. Using again compactness of $K$ we can choose $N_{0}$  in $\nat$ and $(R_{n})$ in $ \prod_{N_{0}\le n}\cC^{G}$ sufficiently large such that the factorization $f^{\prime}$ exists and the diagram commutes.  	

We now consider the projection (see \eqref{brtboijogr4g45g5})
\[p:E(\coprod_{n \in \IN}^{\free} \cO(P_{(R_n)_{Y}}(Y))) \to E(\cO(P_{(R_{N_0+1})_{Y}}(Y)))\] arising from excision and the inclusion of the summand with index $N_{0}+1$. The composition of $f'$ with the canonical map \[E(\coprod_{n \leq N_0}^{\free} \cO(P_{R_{Y}}(Y))) \to E(\coprod_{n \in \IN}^{\free} \cO(P_{R_{Y}}Y)))\] and $p$ is equivalent to zero. Hence the composition $p\circ !!!\circ \tr_{\nat} \circ \tilde f$ is equivalent to zero. On the other hand, by property \eqref{bvtrb4g4tb}  of the transfer, the latter  is equivalent to 
  \[  K \xrightarrow{\tilde f}E(\cO(P_{R_{Y} }(Y)))\to E(\cO(P_{R_{N_0+1}}(Y)))\ .\] 
  This implies that $f$ is equivalent to zero.
\end{proof}

We now consider a sequence $(X_{n})$ of $G$-bornological coarse spaces. 
 
 \begin{lem}\label{lem:componentvanishing-sequencevanishing}
If the $G$-bornological coarse space $X_n$ is $E$-vanishing for every $n$ in $\IN$, then the sequence of $G$-bornological coarse spaces $(X_n) $ is $E$-vanishing.
\end{lem}
\begin{proof}
We consider a sequence $(Y_{n})$ of invariant subspaces of $(X_{n})$.
We must show that $\bC^{\infty}((Y_{n}))$ is a phantom object.

By \cref{ef988u9u89423r} it suffices to show that $\bC_{\free}^{\infty}((Y_{n}))$ is a phantom object.
  
Let $K$ be a compact object of $\bC$ and \[f \colon K \to \colim_{N\in \nat}\colim_{(S_n)\in \prod_{n \in \nat}\cC^{G}(X_{n})}E(\coprod^{\free}_{N\le n}\cO(P_{(S_n)_{Y_{n}}}(Y_n)))\] be 
some  morphism. We must show that $f$ is equivalent to zero.

 Since $K$ is compact there exists a factorization
\[\xymatrix{K\ar[rr]^-{f}\ar@{..>}[dr]^-{\tilde f}&& \colim\limits_{N\in \nat} \colim\limits_{(S_n)\in \prod_{n \in \nat}\cC^{G}(X_{n})} E(\coprod\limits^{\free}\limits_{N\le n}\cO(P_{(S_n)_{Y_{n}}}(Y_n)))\\&E(\coprod\limits^{\free}\limits_{N_{0}\le n}\cO(P_{(R_n)_{Y_{n}}}(Y_n)))\ar[ur]&}\]
for some sufficiently large choices of $N_{0}$ in $\nat$  and $(R_{n})$ in $ \prod_{n \in \nat}\cC^{G}(X_{n})$.

  By our assumption on the $G$-bornological coarse spaces $X_{k}$ for each $k$ in $\nat$ with $N_{0}\le k$   there exists   $R_k' $ in $ \cC^G(X_k)$ such that the composition of $\tilde f$ with the projection \[E(\coprod^{\free}_{N_{0}\le n} \cO(P_{(R_n)_{Y}}(Y_n))) \to E(\cO(P_{(R_k)_{Y}}(Y_k)))\to E(\cO(P_{(R_k')_{Y}}(Y_k)))\]
 is equivalent to zero, where the first morphism is the projection onto the $k$th summand.
   By weak additivity of $E$, it follows that the map \[K \xrightarrow{\tilde f} E(\coprod^{\free}_{N_{0}\le n} \cO(P_{(R_n)_{Y}}(Y_n))) \to E(\coprod^{\free}_{N_{0}\le n} \cO(P_{(R_n')_{Y}}(Y_n)))\] is also trivial. 
   This implies that $f$ is equivalent to zero. \end{proof}
  
Let $X$  be a $G$-bornological space with bornology $\cB$, $Y$ be a $G$-set and $f:Y\to X$ be an equivariant map of sets. 
 \begin{ddd}
The \emph{induced bornology} $f^{-1}\cB$ on $Y$ is defined to be  the minimal bornology on $Y$ such that the map $f$ is proper. 
 \end{ddd}

Let $X$ be a $G$-bornological coarse space and $(U_{i})_{i\in I}$ be an equivariant family of subsets. 

\begin{ddd}\label{eoifoewfewfewfwe}
We define the $G$-bornological coarse space
$\coprod_{i\in I}^{\sub}U_{i}$ as follows:
\begin{enumerate}
\item The underlying $G$-set is $\coprod_{i\in I} U_{i}$.
\item The coarse structure is the one defined in \cref{ioefoweifwefwef}.
\item The bornology is induced from $X$ via  the canonical map $\coprod_{i\in I} U_{i}\to X$.\qedhere
\end{enumerate}
   \end{ddd}

  \begin{rem}
  In the situation of \cref{eoifoewfewfewfwe} we have a morphism of $G$-bornological coarse spaces $\coprod_{i\in I}^{\sub}U_{i}\to X$. The bornology on the domain of that map is induced from $X$, but the coarse structure is in general smaller than the induced coarse structure. 
  \end{rem}

We consider a class $\FP$ of sequences of $G$-bornological coarse spaces.
Let $(X_{n})$ be a sequence of $G$-bornological coarse spaces.   

\begin{ddd}
	\label{def:strongdec}
The sequence $(X_{n})$ 
  is \emph{decomposable over $\FP$}  if for every sequence $(T_n)$ in $
     \prod_{n\in \nat}\cC^G_{n} $  and  for  all natural numbers $n$  there exist pairwise $T_{n}$-disjoint and nice equivariant families $\cU^{T_{n}}:=(U_{i}^{T_{n}})_{i\in I_{n}}$ and $\cV^{T_{n}}:=(V^{T_{n}}_{i})_{j\in J_{n}}$ of subsets of $X$
     such that: 
     \begin{enumerate} 
   \item $X_{n}=U^{T_{n}}\cup V^{T_{n}}$ with
      $U^{T_n}:=\bigcup_{i\in I_n}U_i^{T_n}$ and $V^{T_n}:=\bigcup_{j\in J_n}V_j^{T_n}$.      
 \item  The sequences	of $G$-bornological coarse spaces \[\left(\coprod_{i\in I_n}^{\sub}U_i^{T_n}\right) \ , \quad \left(\coprod_{j\in J_n}^{\sub}V_j^{T_n}\right) \]
	belong to $\FP$.  \qedhere
	\end{enumerate}
 
\end{ddd}
Let $(X_{n})$ be a sequence of $G$-bornological coarse spaces and $(Y_{n})$ be a sequence of subspaces.
\begin{lem} \label{rioeoirgergergerg}
If $\FP$ is closed under taking sequences of subspaces and  the sequence $(X_{n})$ is decomposable over $\FP$, then    $(Y_{n})$ is decomposable over $\FP$.
\end{lem}
\begin{proof}
We use the notation appearing in \cref{def:strongdec}. For every natural number $n$ we can interpret  invariant entourages $T_{n}$ of $Y_{n}$ as invariant entourages of $X_{n}$.
We set
\[U^{T_{n}}_{Y,i}:=U_{i}^{T_{n}}\cap Y_{n}\ , \quad V^{T_{n}}_{Y,j}:=V^{T_{n}}_{j}\cap Y\ .\] In the following 
we use the notation $\cU^{T_{n}}_{Y}:=(U^{T_{n}}_{Y,i})_{i\in I}$. We   observe by  an inspection of the definitions  that we have an isomorphism
\[\coprod_{i\in I_n}^{\sub_{Y }}U_{Y,i}^{T_n} \cong \coprod_{i\in I_n}^{\sub_{X }}U_{Y,i}^{T_n}\]
of $G$-bornological coarse spaces. Here
  the subscript $Y $ or $X$ at the $\sub$-symbol indicates that the coarse structures are generated 
by the entourages $\sub^{\cU^{T_{n}}_{Y}}((S_{n})_{Y})$  (or $\sub^{\cU_{Y}^{T_{n}}}(S_{n})$, respectively)  for all entourages $S_{n}$ of $X_{n}$, and that the bornologies are induced from the bornologies of $Y_{n}$ (or $X_{n}$, respectively).
We conclude that
$\left(\coprod_{i\in I_n}^{\sub_{Y}}U_{Y,i}^{T_n}\right)$ is a sequence of subspaces of the sequence  $\left(\coprod_{i\in I_n}^{\sub_{X}}U_{i}^{T_n}\right)$ and hence belongs to $\FP$ by assumption.
 
 A similar reasoning applies to the $V$-sequences.
\end{proof}

Recall that $\cV_{E}$ and $\VP_{E}$ denote the classes of $E$-vanishing $G$-bornological coarse spaces (\cref{opwefwefewfwf}) and of $E$-vanishing sequences of $G$-bornological coarse spaces (\cref{roigrgegg}).

Let $X$ be a $G$-bornological coarse space with coarse structure $\cC$,
and let $\underline{X}$ be the corresponding  constant sequence of $G$-coarse spaces.

\begin{kor}\label{cor:constantdecomposableseq}
If $X$ is decomposable over $\cV_{E}$, then the constant sequence $\underline{X} $ is decomposable over $\VP_{E}$.
\end{kor}
\begin{proof}
 Let $(T_n)$ in $\prod_{n\in \nat} \cC^{G} $ be given.  
 Since $X$ is decomposable over $\cV_{E}$,   for each natural number $n$  there exist
   pairwise $T_{n}$-disjoint nice equivariant   families $(U_{i}^{T_{n}})_{i\in I_{n}}$ and $(V^{T_{n}}_{i})_{i\in J_{n}}$ of subsets of $X$  such that
 $X = U^{T_{n}} \cup V^{T_{n}}$ with   \[U^{T_n}:=\bigcup_{i\in I_n}U_i^{T_n},\quad V^{T_n}:=\bigcup_{j\in J_n}V_j^{T_n}\ ,\] and such that the $G$-bornological coarse spaces 
 $\coprod_{i \in I_n}^{\sub} U^{T_{n}}_{i}$ and $\coprod_{j \in J_n}^{\sub} V^{T_{n}}_{j}$ are $E$-vanishing. By \cref{lem:componentvanishing-sequencevanishing}, both sequences of  $G$-bornological coarse spaces 
     $(\coprod_{i \in I_n}^{\sub} U^{T_{n}}_{i})$ and $(\coprod_{j \in J_n}^{\sub} V^{T_{n}}_{j})$ are $E$-vanishing sequences. We conclude that the sequence $\underline{X}$ 
    is decomposable over $\VP_{E}$.
\end{proof}

\subsection{Properties of the Rips complex}
\begin{rem}
As a preparation of what follows we recall the following conventions.

If $Y$ is a $G$-set with a $G$-invariant quasi-metric $d$, then for every $r$ in $\R$ we define the invariant entourage
\[U_{r}:=\{(x,y)\in Y\times Y\:|\: d(x,y)\le r\}\ .\]

The $G$-coarse structure $\cC_{d}$ on $Y$ induced by the metric is the coarse structure generated by the enourages $U_{r}$ for all $r$ in $\R$.

All this applies in particular to the quasi-metric $G$-space $P^{X}_{S}(Y)$ for a $G$-bornological coarse space $X$ and invariant subset $Y$ of $X$ and invariant entourage $S$ of $X$.
\end{rem}

Let $X$ be a $G$-bornological coarse space, $Y$ be an invariant subset of $X$, and $S$ be an invariant entourage of $X$ such that $\diag(X)\subseteq S$. Note that this implies that  $S^{k}\subseteq S^{k+1}$ for all integers $k$ and that $(S^{n}[Y])_{n\in \nat}$ is an increasing family of invariant subsets of $X$.
 
\begin{lem}\label{lem:ballsinripscomplexes}
	For every integer $k$ 
	we have the inclusion 
	\[U_{k-2}[P_S^X(Y)]\subseteq P_S^X(S^k[Y])\subseteq U_{k+1}[P_S^X(Y)]\ .\]  In particular, $(P_S^X(S^n[Y]))_{n\in \IN}$ is a big family in the $G$-bornological coarse space $P_{S}(X)$.
\end{lem}
\begin{proof} 
In the following argument we identify the zero skeleton of $P_{S}(X)$ with $X$. This is possible since  by assumption $S$ contains the diagonal.
 
Let $\mu$ be a point in $U_{k-2}[P_S^X(Y)]$. Then it is contained in some simplex. We choose   some vertex $x$ of that simplex. Then $x\in U_{k-1}[P_S^X(Y)]$. The shortest path which connects $x$ with a point in $P_S^X(Y)$ is contained in the one-skeleton. Hence we can conclude that $x\in S^{k-1}[Y]$. But then all vertices of the simplex containing $\mu$ are contained in $S^{k}[Y]$. Hence $\mu\in P_S^X(S^k[Y])$. 

Let now $\mu$ be a point in $P_S^X(S^k[Y])$.
We again choose a vertex $x$ of the simplex containing $\mu$. Then  the distance of $x$ from $P_{S}^{X}(Y)$ is  bounded by $k$. Hence the distance of $\mu$ from
$P_{S}^{X}(Y)$ is  bounded by  $k+1$, i.e., we have $\mu\in U_{k+1}[P_S^X(Y)]$.
\end{proof}

Let $X$ be a $G$-bornological coarse space and $\cV:=(V_{i})_{i\in I}$ be an equivariant pairwise disjoint family of subsets. Let $S$ and $T$ be invariant entourages of $X$ such that $S$ contains   the diagonal, and set $S^{\prime}:=S\cup \sub^{\cV}(T)$ (\cref{ioefoweifwefwef}).
  
Let $n$ be a natural number.

\begin{lem}\label{lem:distancesinrelativerips}

 If $(P_{S}^{X}(V_{i}))_{i\in I}$ is $U_{n}$-disjoint in $P^{X}_{S}(X)$, then  
$(P_{S^{\prime}}^{X}(V_{i}))_{i\in I}$ is $U_{n-2}$-disjoint in $P^{X}_{S^{\prime}}(X)$.
	 
\end{lem}
\begin{proof}
Since $S^{\prime}$ contains the diagonal
 we can identify $X$ with the zero skeleton of the Rips complex  $P^{X}_{S^{\prime}}(X)$.  
Since every point in   the Rips complex   has distance at most one to a point in the 0-skeleton, it suffices to show that the family $(V_{i})_{i\in I}$ is $U_{n}$-disjoint in $P^{X}_{S^{\prime}}(X)$.

We argue by contradiction.	
Let $d$ denote the distance in $	P_{S^{\prime}}(X)$. Assume that $i$ and $j$ belong to $I$ such that $i\not= j$ and $d(V_{ i},V_{j })<n$. Then there exists a sequence $(x_{k})_{k=0,\dots,n-1}$ in $X$
with $x_0\in V_{i }$, $x_{n-1}\in V_{j }$, and $(x_k,x_{k+1})\in S'$ for all $k$ in $ \{0,\dots,n-2\}$. 

We let  $k_{0}$ in $  \{0,\dots,n-1\}$ be the minimal element such that   $x_{k_{0}} \in V_{i}$ and $x_{k_{0}+1} \notin V_{i }$. Such an element $k_{0}$ exists since $V_{i}\cap V_{j}=\emptyset$.

We  let  $k_{1} $ in  $\{0,\dots,n-1\}$ be the  maximal element such that for all $m$
in  $ \{k_{0},\dots,k_{1}-1\}$ we have $(x_{m},x_{m+1})\not\in \bigcup_{i\in I}(T\cap (V_i\times V_i))$.  We have $k_{0}<k_{1}\le n-1$ again since the family $(V_{i})_{i\in I}$ is pairwise disjoint. There exists $l$ in $I$ such that $x_{k_{1}}\in V_{l  }$.
Hence $d(V_{i },V_{l  })\le k_{1}-k_{0}\le n-1$.
This contradicts the assumption that the family $(V_{i })_{i\in I}$ is $U_{n}$-disjoint.
\end{proof}

Let $X$ be a $G$-bornological coarse space, let $U$ be an invariant subspace and let $k$ be a natural number. We consider again invariant entourages $S$ and $T$ of $X$ containing the diagonal  and form the invariant entourage 
\[S':=S\cup (T\cap (S^k[U]\times S^k[U]))\ .\]

 \begin{lem}\label{iuhgeirughergiregergergregregergrer}
 We have  $(S')^m[U]\subseteq S^{k+m}[U]$.
\end{lem}
\begin{proof}
We consider a point $x$ in $(S')^m[U]$. Then there exist a sequence 
$(x_0,\ldots,x_m)$ in $X$ with $x_0\in U$, $x_m=x$ and $(x_i,x_{i+1})\in S'$ for all $i$ in $\{0,\dots,m-1\}$. Let $l$ in $ \{0,\dots,m-1\}$ be maximal with $x_l\in S^k[U]$. By definition of $S'$ we have $(x_i,x_{i+1})\in S$ for all  $i$ in $\{l+1,\dots,m-1\}$. Thus $x\in S^{m-l}[S^k[U]]\subseteq S^{m+k}[U]$.
\end{proof}

\subsection{Vanishing is closed under decomposition}

Recall that $\VP_{E}$ denotes the class of $E$-vanishing sequences (\cref{roigrgegg}).
\begin{theorem}
	\label{thm:dec}
	The class $\VP_{E}$ is closed under decomposition.
\end{theorem}
This subsection is devoted to the proof of this theorem.

Suppose that the sequence $(X_n)$ is a sequence of $G$-bornological coarse spaces which is decomposable over $\VP_{E}$ (see \cref{def:strongdec}). In view of \cref{roigrgegg} we  have to show  for every sequence $(Y_n)$ of $G$-invariant subsets that $\bC^{\infty}((Y_{n}))$ defined in \eqref{gerg43r4r3}
is a phantom object.

It immediately follows from   \cref{roigrgegg}  that the class
 $\VP_{E}$ is closed under taking sequences of $G$-invariant subspaces. By \cref{rioeoirgergergerg} every sequence of $G$-invariant subspaces $(Y_n)$ is also decomposable over $\VP_{E}$.
So we must see that the decomposability of the sequence $(Y_{n})$ implies that it is $E$-vanishing.
 
  It therefore suffices to show that $\bC^{\infty}((X_{n}))$ is a phantom.   Note that  \[
\bC^{\infty}((X_{n}))  \simeq \colim_{r\in (0,\infty)} \colim_{N\in \nat} \colim_{(S_{n})\in \prod_{n\in \nat }\cC_{n}^{G}}E(\coprod_{ N\le n }^{\semi(r)}\cO(P_{S_n}(X_n)))\ .\]

Let $K$ be a compact  object in $\bC$ and 
\[f:K\to \colim_{r\in (0,\infty)} \colim_{N\in \nat} \colim_{(S_{n})\in \prod_{n\in \nat }\cC_{n}^{G}}E(\coprod_{ N\le n }^{\semi(r)}\cO(P_{S_n}(X_n)))\]
be a morphism. By compactness of $K$ it factorizes over a morphism  \[\tilde f:K\to   \colim_{N\in \nat} E(\coprod_{ N\le n }^{\semi(r)}\cO(P_{S_n}(X_n)))\]
for some real number $r$, and a sequence $(S_{n}) $ (we will assume that  $S_{n}$ contains the diagonal of $X_{n}$ for every integer $n$) in $ \prod_{ n\in \nat }\cC_{n}^{G}$. It suffices to show that there exists  a real number $r^{\prime\prime}$ with $r\le r^{\prime\prime}$, and a sequence
$(S^{\prime\prime}_{n}) $ in $ \prod_{ n\in \nat }\cC_{n}^{G}$ with $(S_{n}) \le (S_{n}^{\prime\prime})$ such that the induced (by the inclusion of Rips complexes) map \begin{equation}\label{f349u0ff34f3f}
\tilde f^{\prime\prime}:K\to     \colim_{N\in \nat} E(\coprod_{ N \le n }^{\semi(r^{\prime\prime})}\cO(P_{S^{\prime\prime}_n}(X_n)))
\end{equation}
  is equivalent to zero.
 
Let $(U_{n})$ and $(V_{n})$ be sequences of invariant subsets of the sequence $(X_{n})$ such that for every natural number $n$ we have $X_{n}=U_{n}\cup V_{n}$.
Using the second assertion of \cref{lem:ballsinripscomplexes} and excision for $E$ we obtain a pushout:
\begin{equation}\label{eq:decompositionpushout}
\hspace{-1cm} \xymatrix{
	\colim\limits_{N,k\in \nat} E(\coprod\limits_{ N\le n }^{\semi(r)} \cO(P_{S_n}^{X_n}(S_n^k[U_n]\cap S_n^k[V_n])))\ar[r]\ar[d]&\colim\limits_{N,k\in \nat}E(\coprod\limits_{ N\le n }^{\semi(r)}\cO (P_{S_n}^{X_n}(S_n^k[U_n])))\ar[d]\\
	\colim\limits_{N,k\in \nat}E (\coprod\limits_{ N\le n }^{\semi(r)}\cO (P_{S_n}^{X_n}(S_n^k[V_n])))\ar[r]&\colim\limits_{N\in \nat} E(\coprod\limits_{ N\le n }^{\semi(r)}\cO(P_{S_n}(X_n)))}
\end{equation}

\begin{rem}
At this this point it is important to work with the semi-free union (see \cref{regeriogregregerg}).  Indeed, in general  the family of invariant subsets 
$\left(\coprod_{ N\le n }\cO (P_{S_n}^{X_n}(S_n^k[V_n]))\right)_{k\in \nat}$
is not a big family in the $G$-bornological coarse space  $\coprod_{ N\le n }^{\free}\cO(P_{S_n}(X_n))$
\end{rem}

By the decomposability assumption on $(X_{n})$ 
for every natural number $n$  we can choose nice  equivariant $S_{n}^{n}$-disjoint families
  $(U_{n,i})_{i \in I_n}$ and $(V_{n,j})_{j \in J_n}$ of subspaces of $X_{n}$ such that  $X_{n}=U_{n}\cup V_{n}$ with \[U_{n}:=\bigcup_{i\in I_{n}}U_{n,i} \ ,\quad 
  V_{n}:= \bigcup_{j\in J_{n}}V_{n,j}\] and such that
 the sequences \[\left(\coprod_{i \in I_n}^{\sub} U_{n,i}\right) \ , \quad \left(\coprod_{j \in J_n}^{\sub} V_{n,j}\right) \]  belong to  $\VP_{E}$. This choice will be fixed for the rest of the subsection.

Since above we have chosen $S_{n}^{n}$-disjoint families (the $n$th power is important), by the first assertion of \cref{lem:ballsinripscomplexes} for every  $k$ in $\nat$  there exists an   $N_{1}(k)$ in $\nat$     such that for every integer $n$ with $N_{1}(k)\le n$ the families
\begin{equation}\label{it43toi3t3t34t}
(S_{n}^{k}[U_{n,i}])_{i\in I_{n}}\ , \quad (S_{n}^{k}[V_{n,j}])_{j\in J_{n}}
\end{equation} are $ S_{n}$-disjoint, and the families 
\begin{equation}\label{frwelkmlfwefewf}(P^{X_{n}}_{S_{n}}(S_{n}^{k}[U_{n,i}]))_{i\in I_{n}}\ , \quad (P^{X_{n}}_{S_{n}}(S^{k}_{n}[V_{n,j}]))_{j\in J_{n}}\end{equation}
are $U_{r}$-disjoint. Note that $N_{1}(k)$ also depends on $r$, but we will not indicate this in the notation.

For a $G$-uniform bornological coarse space $A$, whose uniform structure is induced by an invariant metric and a real number $r$, we let  $\cO^{r}(A)$ denote the $G$-uniform bornological coarse space
obtained from $\cO(A)$ by replacing the coarse structure by the coarse structure generated by all entourages of propagation (\cref{iofjoewffewfw}) bounded by $r$.

We abbreviate  \[Y^{k}_{n,i,j} := S_n^k[U_{n,i}]\cap S_n^k[V_{n,j}]\]
and consider the family \[\cY^{k}_{n}:=(Y^{k}_{n,i,j})_{i\in I_{n},j\in J_{n}}\ .\]
Below, we use the abbreviation \begin{equation}\label{f34f3pokpo34f34f36456}
  \sub^{k}(S_{n}):=\sub^{\cY^{k}_{n}}(S_{n})\ ,\end{equation} see \cref{ioefoweifwefwef}.

 \begin{lem}\label{ergioergerg345435}
For  all natural numbers $n$  and $k$ with $N_{1}(k)\le n$ we have an isomorphism of $G$-bornological coarse spaces     
\begin{equation}\label{vwerkhwekfdewqfwf}
\cO^{r}(P^{X_{n}}_{S_{n}}(S_n^k[U_n]\cap S_n^k[V_n]))\cong \cO^{r}(P_{\sub^{k}(S_{n})}( \coprod_{i\in I_{n},j\in J_{n}}^{\sub}Y^{k}_{n,i,j}))\ .
\end{equation}

\end{lem}
\begin{proof} 
First of all the underlying $G$-simplicial complexes 
$P^{X_{n}}_{S_{n}}(S_n^k[U_n]\cap S_n^k[V_n])$ and
$P_{\sub^{k} (S_{n})}( \coprod_{i\in I_{n},j\in J_{n}}^{\sub}Y^{k}_{n,i,j})$ are isomorphic  since the families \eqref{it43toi3t3t34t} are $S_{n}$-disjoint. The distance of the different components on the right-hand side in the metric of $P_{S_{n}}^{X_{n}}(X_{n})$ is bigger than $r$ since the families \eqref{frwelkmlfwefewf} are $U_{r}$-disjoint. 
Hence the entourages on $P^{X_{n}}_{S_{n}}(S_n^k[U_n]\cap S_n^k[V_n])$ and on $P_{\sub^{k} (S_{n})}( \coprod_{i\in I_{n},j\in J_{n}}^{\sub}Y^{k}_{n,i,j})$ of propagation less or equal to $r$ (measured in the respective quasi-metric) are equal (under the natural identification).
The bornologies of both spaces coincide by construction.

Since we consider the cone $\cO^{r}$ we  can conclude that the natural map  \eqref{vwerkhwekfdewqfwf} induces an isomorphism of $G$-bornological coarse spaces.
 \end{proof}

Using \cref{ergioergerg345435} we can rewrite the pushout square \eqref{eq:decompositionpushout} in the form:
\[ \xymatrix{
	 \colim\limits_{N,k\in \nat}E (\coprod\limits_{ N \le n }^{\semi(r)} \cO(P_{\sub^{k}(S_n)}(\coprod\limits_{i\in I_{n},j\in J_{n}}^{\sub}Y^{k}_{n,i,j})))\ar[r]\ar[d]&  \colim\limits_{N,k\in \nat}E (\coprod\limits_{ N \le n }^{\semi(r)}\cO (P_{S_n}^{X_n}(S_n^k[U_n])))\ar[d]\\
	  \colim\limits_{N ,k \in \nat}E (\coprod\limits_{ N \le n }^{\semi(r)}\cO (P_{S_n}^{X_n}(S_n^k[V_n])))\ar[r]&  \colim\limits_{N \in \nat}E(\coprod\limits_{ N \le n }^{\semi(r)}\cO(P_{S_n}(X_n)))}
\]
Since we take the colimit over $N $  the fact that the isomorphism
\eqref{vwerkhwekfdewqfwf} exists only for suffiently large $n$ (and the condition of being sufficiently large depends on $k$) does not cause any problem.

Let
\[ \partial \colon   \colim_{N \in \nat}E(\coprod_{ N \le n }^{\semi(r)}\cO(P_{S_n}(X_n))) \to  \colim_{N,k\in \nat}\Sigma E (\coprod_{ N \le n }^{\semi(r)} \cO(P_{\sub^{k}(S_n)}(\coprod_{i\in I_{n},j\in J_{n}}^{\sub}Y^{k}_{n,i,j}))) \]
denote the boundary map of the Mayer--Vietoris sequence.
Using compactness of $K$, the morphism $\partial \circ \tilde f$ factorizes over a morphism
\[g:K\to    \colim_{N\in \nat} E(\coprod_{ N \le n }^{\semi(r)} \cO(P_{\sub^{k}(S_n)}(\coprod^{\sub}_{i\in I_{n},j\in J_{n}}Y^{k}_{n,i,j})))\]
for some $k$ in $\nat$ which will be fixed from now on.

\begin{lem}
	\label{lem:1}
	There exist  a real number $r'$ such that $r^{\prime}\geq r$ and a sequence $(T_n)$ in $\prod_{n\in \nat}\cC^{G}$ such that   with $(S_{n}^{\prime}):=(S_{n}\cup \sub^{k}(T_{n}))$ the map
	\[K\to  \colim_{N\in \nat}E(\coprod_{ N \le n }^{\semi(r^{\prime})} \cO(P_{\sub^{k}( S^{\prime}_n)}(\coprod_{i\in I_{n},j\in J_{n}}^{\sub}Y^{k}_{n,i,j})))\] induced by $g$
	is equivalent to zero. \end{lem}
\begin{proof}
As in \cref{ergioergerg345435}
for all integers $n$   with $N_{1}(k)\le n$ we also have an isomorphism 
\begin{equation}\label{vwerkhwekfdewqfwf1}
\cO^{r}(P_{\sub^{k,\prime}(S_{n})}^{\coprod_{i\in I_{n}}^{\sub} S_{n}^{k}[U_{n,i}]  }  (\coprod_{i\in I_{n}}S_{n}^{k}[U_{n,i}] \cap S_n^k[V_{n}]))\cong \cO^{r}(P_{\sub^{k}(S_{n})}( \coprod_{i\in I_{n},j\in J_{n}}^{\sub} Y^{k}_{n,i,j}))\ ,
\end{equation}
where we use the abbreviation \begin{equation}\label{fef235345453e2e32e2e1}
\sub^{k,\prime}(S_{n}):=\sub^{\cU^{k}_{n}}(S_{n})\,  \end{equation} with
$\cU_{n}^{k}:=(S^{k}_{n}[U_{n,i}])_{i\in I_{n}}$,  and we use the  equality
\[S^{k}_{n}[V_{n}]=\bigcup_{j\in J_{n}} S_{n}^{k}[V_{n,j}]\]
of subsets of $X$.

	The sequence $( \coprod^{\sub}_{i\in I_{n}} S_n^k[U_{n,i}]\cap S_n^k[V_{n}]))_{N_{1}(k)\le n}$ of $G$-bornological coarse spaces  is a  sequence of subspaces of the sequence
	$(\coprod_{i\in I_{n}}^{\sub} S_n^k[U_{n,i}])_{N_{1}(k)\le n}$. Since we assume  that the family $(U_{n,i})_{i\in I_{n}}$ is nice 
	 		the bornological coarse space $\coprod_{i\in I_{n}}^{\sub}S_n^k[U_{n,i}]$ is equivalent to $\coprod_{i\in I_{n}}^{\sub}U_{n,i}$.
	 
	 	By assumption the sequence $(\coprod_{i\in I_{n}}^{\sub} [U_{n,i}])$ is $E$-vanishing. By \cref{lem:vanishingiscoarselyinvariant}  we can now conclude that the sequence 
	 	$(\coprod_{i\in I_{n}}^{\sub}S_n^k[U_{n,i}])$ is $E$-vanishing.
		This implies that the sequence   $( \coprod^{\sub}_{i\in I_{n}}S_{n}^{k}[U_{n,i}] \cap S_n^k[V_{n}])$ is $E$-vanishing.
	 Hence using \cref{wrowpfewfwf}, there exist a real number  $r'$ such that $r\le r^{\prime}$ and a sequence  $(T_n)$ in $\prod_{n\in \nat}\cC_{n}^{G}$ with $ (S_{n})\le (T_{n})$ such that
	\[K \xrightarrow{g}    \colim_{N\in \nat} E (\coprod_{ N \le n }^{\semi(r )} \cO(P_{\sub^{k}( S_n)}(\coprod_{i\in I_{n},j\in J_{n}}^{\sub}Y^{k}_{n,i,j}))) 
	\to  \colim_{N\in \nat} E(\coprod_{ N \le n }^{\semi(r^{\prime})} \cO(P_{\sub^{k}( T_n)}(\coprod_{i\in I_{n},j\in J_{n}}^{\sub}Y^{k}_{n,i,j}))) \]
	is equivalent to zero. We now note that
	$\sub^{k}(T_{n})=\sub^{k}(S_{n}^{\prime})$.
\end{proof}

It follows from \cref{lem:1} that \begin{align*}
K\xrightarrow{\partial\circ \tilde f}& \colim_{N,k\in \nat}\Sigma E(\coprod_{ N \le n }^{\semi(r)} \cO(P_{\sub^{k}(S_n)}(\coprod_{i\in I_{n},j\in J_{n}}^{\sub}Y^{k}_{n,i,j})))\\
\to& \colim_{N,k^{\prime}\in \nat}\Sigma E (\coprod_{ N \le n }^{\semi(r^{\prime})} \cO(P_{\sub^{k^{\prime}}(S^{\prime}_n)}(\coprod_{i\in I_{n},j\in J_{n}}^{\sub}Y^{k^{\prime}}_{n,i,j})))\end{align*}
 is equivalent to zero. 
Consequently,  using the  naturality of the Mayer--Vietoris sequences and the compactness of $K$ (in order to be able to choose $k^{\prime}$), the composition 
\[K\xrightarrow{\tilde f^{\prime}}   \colim_{N\in \nat}E(\coprod_{ N\le n }^{\semi(r)}\cO(P_{S_n}(X_n)))\to \colim_{N\in \nat} E(\coprod_{ N\le n }^{\semi(r^{\prime})}\cO(P_{S^{\prime}_n}(X_n)))\]  
(where $r^{\prime}$ and $(S_{n}^{^{\prime}})$ are as in \cref{lem:1})
lifts to a morphism
\[(\hat f_{U},\hat f_{V}) \colon K \to   \colim_{N\in \nat}E(\coprod_{N\le n}^{\semi(r^{\prime})}\cO(P_{S'_n}^{X_n}(S_n^{k'}[U_n]))) \oplus \colim_{N} E(\coprod_{N\le n}^{\semi(r^{\prime})} (P_{S'_n}^{X_n}(S_n^{k'}[V_n])))\]
for some integer $k^{\prime}$ such that $k^{\prime}\geq k$.

The notation 
$ \sub^{k^{\prime},\prime}(T_{n}) $ is as in \eqref{fef235345453e2e32e2e1},  while $\sub^{k}(T_{n})$ is as in \eqref{f34f3pokpo34f34f36456} for the specific choice of the integer $k$ made just before the statement of \cref{lem:1}.
 
Let $(T_{n}^{\prime\prime})$ be  a sequence in $\prod_{n\in \nat} \cC_{n}^{G}$ such that  $(T_{n})\le(T_{n}^{\prime\prime})$.
We define the sequence \begin{equation}\label{f3roifjo34243f34f}
(S_{n}^{\prime\prime}):=(S_{n}^{\prime}\cup \sub^{k^{\prime},\prime}(T_{n}^{\prime\prime}))\ .
\end{equation}

By \cref{iuhgeirughergiregergergregregergrer} and \cref{lem:distancesinrelativerips}
there exists an integer $N_{2}$ such that $N_{1}(k)\le N_{2}$ and for all integers $n$ with $N_{2}\le n$ the family $(S_{n}^{k^{\prime}}[U_{n,i}])_{i\in I_{n}}
$is $\tilde S_{n}^{\prime\prime}$-disjoint and the family
$(P_{S_{n}^{\prime\prime}}(S^{k^{\prime}}_{n}[U_{n,i}]))_{i\in I_{n}}$ is
$r^{\prime}$-disjoint.

The following is similar to \cref{ergioergerg345435}.

\begin{lem}\label{lem:individualterm}
For all integers $n$  with $N_{2}\le n$ we have an isomorphism of $G$-bornological coarse spaces 
\[\cO^{r^{\prime}}(P^{X_{n}}_{S_{n}^{\prime\prime}}(S_{n}^{k^{\prime}}[U_{n}])) \cong  \cO^{r^{\prime}}(P_{\sub^{k^{\prime},\prime}(S_{n}^{\prime\prime})}(\coprod_{i\in I}^{\sub} S^{k^{\prime}}_{n}[U_{n,i}]))\ .\]
\end{lem}\begin{proof}
First of all the underlying $G$-simplicial complexes 
\[P^{X_{n}}_{S_{n}^{\prime\prime}}(S_{n}^{k^{\prime}}[U_{n}]) \ , \quad
P_{\sub^{k^{\prime},\prime}(S_{n}^{\prime\prime})}(\coprod_{i\in I}^{\sub} S^{k^{\prime}}_{n}[U_{n,i}])\] are isomorphic  since the family $(S_{n}^{k^{\prime}}[U_{n,i}])_{i\in I_{n}}$
  is
$S^{\prime\prime}_{n}$-disjoint. The distance of the different components on the right-hand side in the metric of $P_{S^{\prime}_{n}}^{X_{n}}(X_{n})$ is bigger than $r^{\prime}$. Since we consider the cone $\cO^{r^{\prime}}$ we now can conclude that the $G$-coarse structures
on both sides coincide. The bornologies coincide by construction. 
 \end{proof}

We get an equivalence
\begin{equation}\label{vweoijowepfewfewf}
\colim_{N\in \nat} E(\coprod_{N\le n}^{\semi(r^{\prime})}\cO(P_{S''_n}^{X_n}(S_n^{k'}[U_n]))) \simeq  \colim_{N\in \nat}E(\coprod_{N\le n}^{\semi(r^{\prime})}\cO(P_{\sub^{k^{\prime},\prime}(S_{n}^{\prime\prime})}(\coprod_{i\in I}^{\sub} S^{k^{\prime}}_{n}[U_{n,i}]) )) 
\end{equation}

\begin{lem}
	\label{lem:2}
	 There exist a real number $r^{\prime\prime}$ such that  $r'' \geq r'$ and a sequence   $(S_n'')$ in $\prod_{n\in \nat}\cC_{n}^{G}$ with $(S_n'')\geq (S_n')$ such that the map
	\[K \xrightarrow{\hat f_{U}} \colim_{N\in \nat} E(\coprod_{N\le n}^{\semi(r^{\prime})}\cO(P_{S'_n}^{X_n}(S_n^{k'}[U_n]))) \to
	\colim_{N\in \nat} E(\coprod_{N\le n}^{\semi(r^{\prime\prime})}\cO(P_{S''_n}^{X_n}(S_n^{k'}[U_n]))) \]
	is equivalent to zero.
\end{lem}
\begin{proof}
In view of the equivalence \eqref{vweoijowepfewfewf} we must show that for an appropriate choice of the sequence $(T_{n}^{\prime\prime})$ (which determines the sequence $(S_{n}^{\prime\prime})$ by \eqref{f3roifjo34243f34f}) the morphism
\[\mathclap{
K \to \colim_{N\in \nat} E(\coprod_{N\le n}^{\semi(r^{\prime})}\cO(P_{\sub^{k^{\prime},\prime}(S_{n}^{\prime})}(\coprod_{i\in I}^{\sub} S^{k^{\prime}}_{n}[U_{n,i}]) ))  \to
	\colim_{N\in \nat}E(\coprod_{N\le n}^{\semi(r^{\prime\prime})}\cO(P_{\sub^{k^{\prime},\prime}(S_{n}^{\prime\prime})}(\coprod_{i\in I}^{\sub} S^{k^{\prime}}_{n}[U_{n,i}]) ))
}\]
	is equivalent to zero.
As in the proof of  \cref{lem:1} we argue that the sequence of $G$-bornological coarse spaces  $(\coprod_{i\in I}^{\sub} S^{k^{\prime}}_{n}[U_{n,i}])$
is $E$-vanishing. Using \cref{wrowpfewfwf}  
 and compactness of $K$  
 we find the desired real number $r^{\prime\prime}$ and sequence $(S^{\prime\prime}_{n})$.
\end{proof}

Using \cref{lem:2}, we find a real number $r''$ with $r'' \geq r'$ and a sequence $(S_n'')$ in $\prod_{n\in \nat}\cC_{n}$ with $(S_n'')\geq (S_n')$ such that the map \[K \to\colim_{N\in \nat} E(\coprod_{N\le n}^{\semi(r^{\prime\prime})}\cO(P_{S''_n}^{X_n}(S_n^{k'}[U_n])))\oplus \colim_{N\in \nat} E(\coprod_{N\le n}^{\semi(r^{\prime\prime})}\cO(P_{S''_n}^{X_n}(S_n^{k'}[V_n])))\]
  induced by $(\hat f_{U},\hat f_{V})$ and the natural inclusions is equivalent to zero.

Composition of the induced map with the morphism
\begin{eqnarray*} \lefteqn{\hspace{-3cm}
\colim_{N\in \nat} E(\coprod_{N\le n}^{\semi(r^{\prime\prime})}\cO(P_{S''_n}^{X_n}(S_n^{k'}[U_n])))\oplus \colim_{N\in \nat}E(\coprod_{N\le n}^{\semi(r^{\prime\prime})}\cO(P_{S''_n}^{X_n}(S_n^{k'}[V_n])))}&&\\&\to& \colim_{N\in \nat} E(\coprod_{N\le n}^{\semi(r^{\prime\prime})}
\cO(P_{S''_n}(X_{n})))
\end{eqnarray*}

from the Mayer--Vietoris sequence gives the map \eqref{f349u0ff34f3f}
\[\tilde f^{\prime} : K \xrightarrow{\tilde f}  \colim_{N\in \nat} E(\coprod_{ N\le n }^{\semi(r)}\cO(P_{S_n}(X_n))) \to  \colim_{N\in \nat}E(\coprod_{ N\le n }^{\semi(r^{\prime\prime})}\cO(P_{S^{\prime\prime}_n}(X_n))) \]
which is then also equivalent to zero.

 This finishes the proof of \cref{thm:dec}.

\subsection{Proof of Theorem \ref{thm:mainFDC}}\label{sec:proofofmainFDC}

As already observed in \cref{efwoifuiweofewfewff}
the results shown so far imply:

\begin{kor}\label{cor:vanishingclosedunderdecomposition}
 The class $\cV_{E}$ of $E$-vanishing $G$-bornological coarse spaces is closed under decomposition.
\end{kor}

Recall \cref{efiuweofewfewf} of the notion of a semi-bounded $G$-coarse space. Furthermore recall \cref{weffoiwfewwef423453} of a discontinuous $G$-action on a $G$-bornological space.

\begin{prop}\label{prop:semiboundedisvanishing}
 The class of  $G$-bornological coarse spaces whose underlying $G$-coarse space is semi-bounded and has a discontinuous $G$-action is contained in the class of $E$-vanishing spaces $\cV_{E}$.
\end{prop}
\begin{proof} Let $X$ be a $G$-bornological coarse spaces whose underlying space is semi-bounded  and   has a discontinuous $G$-action. Then every 
 invariant subset with the induced $G$-bornological coarse structure has the same properties.
So it suffices to show that
\[E(\cO(P_{S}(X)))\simeq 0\]
for all sufficiently large invariant entourages  $S$ of $X$.

By assumption we can choose an invariant entourage $T$ of $X$ such that every coarse component of $X$ is bounded by $T$. Then for every invariant entourage $S$ of $X$ with $T\subseteq S$ we have an isomorphism
\[P_{S}(X)\cong \bigsqcup_{i\in \pi_{0}(X)} \Delta^{X_{i}}\]
of quasi-metric spaces,
where $X_{i}$ is the component labelled by $i$ and $\Delta^{X_{i}}$ is the set of finitely supported probability measures on $X_{i}$.

 Let $\pi$ be a choice of a subset of $ \pi_{0}(X)$  of representatives of the orbit set $\pi_{0}(X)/G$. For every $i$ in $ \pi$ we  let $G_i$ be the stabilizer of $i$. This group acts simplicially on $X_{i}$. Then we obtain a $G$-equivariant isomorphism of $G$-quasi-metric spaces
	\[ P_S(X) \cong \coprod_{i \in \pi} G \times_{G_i} \Delta^{X_{i}}.\]
	Let $i$ be in $\pi$ and  $x_{i}$ be a point in $X_{i}$.
	The set $Gx_i \cap X_{i}$ is finite since the $G$-action on $X$ is discontinuous and $X_{i}$ belongs to the minimal bornology compatible with the coarse structure of $X$.  Hence, $Gx_i \cap X_{i}$ is the set of vertices of  a simplex $\Delta_i$ in $ \Delta^{X_{i}}$. The barycenter   $b_i$ of $\Delta_{i}$  is fixed under the action of $G_i$. 
	
	There is a unique $G$-equivariant  map
	$b:\pi_{0}(X)\to P_{S}(X)$ such that
	$b(x_{i})=b_{i}$ for all $i$ in $\pi$.
	We consider $\pi_{0}(X)$ as a discrete quasi-metric space.
The map
$b$ is a homotopy equivalence. Indeed, an inverse is given by the map $p:P_{S}(X)\to \pi_{0}(X)$ which sends every point $\mu$ in $P_{S}(X)$ to the coarse component represented by any choice of a vertex of a simplex containing $\mu$. Then $p\circ b=\id_{\pi_{0}(X)}$ and
$b\circ p$ is homotopic to $\id_{P_{S}(X)}$ by a unit-speed   homotopy.
 The cone functor sends this homotopy to a coarse homotopy.
 
 We equip $\pi_{0}(X)$ with the bornology induced from the bornology of $X$ via the map $b$.  
 It follows from the compatibility of the coarse and the bornological structure on $X$ that the projection $p$ is proper.

 We conclude that \[E(\cO(P_{S}(X)))\simeq E(\cO( \pi_{0}(X)))\ .\]
Since $\pi_{0}(X)$ is a discrete $G$-bornological  coarse space the cone
$\cO( \pi_{0}(X))$ is a flasque $G$-bornological coarse space.
 Consequently, $0\simeq E(\cO( \pi_{0}(X)))$.
 This implies \[0
 \simeq E(\cO(P_{S}(X)))\ .\qedhere\]
 \end{proof}

\begin{proof}[Proof of \cref{thm:mainFDC}] 
The theorem is an immediate consequence of \cref{cor:vanishingclosedunderdecomposition} and \cref{prop:semiboundedisvanishing}: The class of $G$-coarse spaces with $G$-FDC is by definition the smallest class of $G$-coarse spaces which is closed under decomposition and contains all semi-bounded spaces. Moreover, any decomposition of $G$-bornological space with discontinuous $G$-action yields families of $G$-bornological coarse spaces with discontinuous $G$-action.
\end{proof}

\bibliographystyle{alpha}
\bibliography{born-transbas}

\end{document}